\documentclass[11pt]{amsart}

\usepackage[a4paper,hmargin=3.5cm,vmargin=4cm]{geometry}
\usepackage{amsfonts,amssymb,amscd,amstext}
\usepackage{graphicx}
\usepackage[dvips]{epsfig}

\usepackage{fancyhdr}
\input xy
\xyoption{all}

\usepackage{times}

\usepackage{enumerate}
\usepackage{titlesec}
\usepackage{mathrsfs}

\titleformat{\section}
{\filcenter\bfseries\large} {\thesection{.}}{0.2cm}{}
\titleformat{\subsection}[runin]
{\bfseries} {\thesubsection{.}}{0.15cm}{}[.]
\titleformat{\subsubsection}[runin]
{\em}{\thesubsubsection{.}}{0.15cm}{}[.]

\usepackage[up,bf]{caption}

\newtheorem{theorem}{Theorem}[section]
\newtheorem{proposition}[theorem]{Proposition}
\newtheorem{claim}[theorem]{Claim}
\newtheorem{lemma}[theorem]{Lemma}
\newtheorem{corollary}[theorem]{Corollary}

\theoremstyle{definition}
\newtheorem{definition}[theorem]{Definition}
\newtheorem{remark}[theorem]{Remark}

\newtheorem{problem}[theorem]{Problem}
\newtheorem{example}[theorem]{Example}

\numberwithin{equation}{section}
\numberwithin{figure}{section}

\newcommand\Ccal{\mathcal{C}}

\newcommand\Lcal{\mathcal{L}}

\newcommand\Pcal{\mathcal{P}}

\newcommand\Ascr{\mathscr{A}}

\newcommand\Cscr{\mathscr{C}}

\newcommand\Lscr{\mathscr{L}}
\newcommand\Oscr{\mathscr{O}}

\newcommand\B{\mathbb{B}}
\newcommand\C{\mathbb{C}}
\newcommand\D{\overline{\mathbb D}}
\newcommand\CP{\mathbb{CP}}
\renewcommand\D{\mathbb D}

\newcommand\N{\mathbb{N}}

\newcommand\R{\mathbb{R}}

\newcommand\T{\mathbb{T}}
\newcommand\Z{\mathbb{Z}}

\renewcommand\c{\mathbb{C}}
\newcommand\cd{\overline{\mathbb D}}

\renewcommand\d{\mathbb D}

\newcommand\n{\mathbb{N}}
\renewcommand\r{\mathbb{R}}

\renewcommand\t{\mathbb{T}}
\newcommand\z{\mathbb{Z}}

\newcommand\igot{\mathfrak{i}}

\renewcommand\igot{\mathfrak{i}}

\newcommand\Ygot{\mathfrak{Y}}

\renewcommand\imath{\igot}

\newcommand\hra{\hookrightarrow}
\newcommand\lra{\longrightarrow}

\newcommand\wt{\widetilde}
\newcommand\wh{\widehat}
\newcommand\di{\partial}
\newcommand\dibar{\overline\partial}

\newcommand\dist{\mathrm{dist}}
\renewcommand\span{\mathrm{span}}
\newcommand\length{\mathrm{length}}

\def\dist{\mathrm{dist}}
\def\span{\mathrm{span}}
\def\length{\mathrm{length}}

\newcommand\Id{\mathrm{Id}}

\usepackage{color}

\begin{document}

\fancyhead[LO]{Holomorphic Legendrian curves}
\fancyhead[RE]{A.\ Alarc\'on, F.\ Forstneri\v c, and F.J.\ L\'opez}
\fancyhead[RO,LE]{\thepage}

\thispagestyle{empty}

\vspace*{1cm}
\begin{center}
{\bf\LARGE Holomorphic Legendrian curves}

\vspace*{0.5cm}

{\large\bf Antonio Alarc\'on, Franc Forstneri\v c, and Francisco J.\  L\'opez}
\end{center}


\vspace*{1cm}

\begin{quote}
{\small
\noindent {\bf Abstract}\hspace*{0.1cm}
In this paper we study holomorphic Legendrian curves in the standard holomorphic contact structure on $\c^{2n+1}$ for any $n\in\n$. 
We provide several approximation and desingularization results which enable us to prove general existence theorems, 
settling some of the open problems  in the subject.
In particular, we show that every open Riemann surface $M$ admits a proper holomorphic Legendrian embedding 
$M\hra\c^{2n+1}$, and we prove that for every compact bordered Riemann surface $M=\mathring M\cup bM$ there exists a 
topological embedding $M\hra \c^{2n+1}$ whose restriction to the interior is a complete holomorphic Legendrian embedding 
$\mathring M\hra \c^{2n+1}$.  As a consequence, we infer that every complex contact manifold $W$ carries relatively 
compact holomorphic Legendrian curves, normalized by any given bordered Riemann surface, which are complete with respect to 
any Riemannian metric on $W$.

\vspace*{0.1cm}

\noindent{\bf Keywords}\hspace*{0.1cm} Riemann surface, complex contact manifold, Legendrian curve.

\vspace*{0.1cm}


\noindent{\bf MSC (2010):}\hspace*{0.1cm} 53D10, 32B15, 32E30, 32H02.}
\end{quote}


\section{Introduction and main results} 
\label{sec:intro}

Let $W$ be a complex manifold of odd dimension $2n+1\ge 3$. 
A holomorphic vector subbundle $\Lscr\subset TW$ of complex codimension one in the tangent bundle $TW$ 
defines a {\em holomorphic contact structure} on $W$ if every point $p\in W$ admits an open neighborhood 
$U\subset W$ such that $\Lscr|_U=\ker\tau$ is defined by a holomorphic $1$-form $\tau$ on $U$ 
satisfying 
\[
	\tau\wedge(d\tau)^n= \tau\wedge d\tau\, \wedge \stackrel{\text{$n$ times}}
	{\cdots}\wedge\, d\tau\neq 0\quad \text{everywhere on $U$}.
\]
This nondegeneracy condition depends only on the subbundle $\Lscr$ and not on the particular choice of 
the local defining $1$-form. The pair $(W,\Lscr)$ is called a {\em complex contact manifold}. When $\Lscr=\ker\tau$ 
for a globally defined holomorphic $1$-form $\tau$ on $W$, we shall write $(W,\tau)$ instead of $(W,\Lscr)$.
A contact subbundle $\Lscr$ is {\em maximally nonintegrable},
meaning that it has no integral complex submanifolds (i.e., tangent to $\Lscr$) of dimension $>n$.
In fact, local holomorphic vector fields tangent to $\Lscr$, along with their first-order commutators, span $TW$ at every point. 
Although the geometry of smooth contact manifolds is a classical subject with large literature devoted to it
(we refer e.g.\ to Geiges' surveys \cite{Geiges2008,Geiges2012} and the references therein), 
many important questions remain open in the holomorphic case.

The most basic example of a complex contact manifold is the complex Euclidean space $\c^{2n+1}$ 
endowed with the {\em standard holomorphic contact form} 
\begin{equation}\label{eq:contact}
	\eta = dz + \sum_{j=1}^n x_j \, dy_j.
\end{equation}
Here,  $(x_1,y_1,\ldots,x_n,y_n,z)$ denote the complex coordinates on  $\c^{2n+1}$.
By Darboux's theorem (see Theorem \ref{th:Darboux-contact} in the Appendix), 
every complex contact manifold $(W^{2n+1},\Lscr)$  is locally contactomorphic to $(\c^{2n+1},\eta)$, 
meaning that in a neighborhood of any point $p\in W$ there are local holomorphic coordinates $(x_1,y_1,\ldots,x_n,y_n,z)$ in which 
$\Lscr=\ker\eta$. 

Let $(W,\Lscr)$ be  a  complex contact manifold of dimension $2n+1$. 
A holomorphic map $F\colon M\to W$ from a complex manifold $M$
is said to be $\Lscr$-{\em Legendrian} if 
\[
	\text{$dF_p(T_p M)\subset \Lscr_{F(p)}$ holds for all points $p\in M$.} 
\]
If $\Lscr=\ker\tau$ for a contact $1$-form $\tau$, then the above condition is equivalent to
\[
	F^*\tau = 0.
\]
This condition is independent of the local parametrization of $M$ and hence can be treated in local holomorphic coordinates on $M$. 
If $F$ is nondegenerate (i.e., an immersion at a generic point of $M$), then $\dim_\C M\le n$ since $\Lscr$ is maximally nonintegrable. 
(When $\dim M<n$, such maps are often referred to as {\em isotropic}; we prefer to use the term Legendrian
even in this {\em subcritical case}.)
The case when $M$ is compact and $W$ is the projective space $\CP^{2n+1}$, endowed with the contact form obtained by projectivizing 
the standard symplectic form of $\c^{2n+2}$ (see Subsec.\ \ref{ss:eta}), has been a major focus of interest in the theory.
An important result in this subject is that every compact Riemann surface embeds as a complex Legendrian curve in $\CP^3$ 
(see Bryant \cite[Theorem G]{Bryant1982JDG} and Segre \cite{Segre1926}). 
On the other hand, when $M$ is an {\em open} Riemann surface, an $\Lscr$-Legendrian 
holomorphic map $F\colon M\to W$ is called an $\Lscr$-{\em Legendrian curve}; if $(W,\Lscr)=(\c^{2n+1},\eta)$ then we shall 
just say that $F$ is a {\em Legendrian curve} in $\C^{2n+1}$. The latter are complex analogues of real Legendrian curves in $\r^{2n+1}$ 
which play a major role in differential geometry.

The aim of this paper is a systematic investigation of holomorphic Legendrian curves in $\c^{2n+1}$ for any $n\in\n$.
In particular, we settle some general questions raised by Alarc\'on and Forstneri\v c in \cite[page 740]{AlarconForstneric2014IM},
as well as a couple of other well-known open problems in the theory. Moreover, as we shall see later in this introduction, our results 
also find applications to holomorphic Legendrian curves in an arbitrary  complex contact manifold. 

The following first main result of the paper concerns the existence of properly embedded Legendrian curves in 
the standard contact manifold $(\C^{2n+1},\eta)$.

%
%
\begin{theorem}[Runge approximation by proper Legendrian embeddings]\label{th:intro-Runge}
Let $M$ be an open Riemann surface and $K\subset M$ be a smoothly bounded compact domain in $M$ 
whose complement has no relatively compact connected components. Then every holomorphic Legendrian curve 
$F\colon K\to\c^{2n+1}$ $(n\in\n)$ can be approximated as closely as desired in the $\Cscr^1(K)$-topology by proper holomorphic 
Legendrian embeddings $\wt F\colon M\hra \c^{2n+1}$. Furthermore, given a pair of indices $\{i, j\}\subset\{1,2,\ldots,2n+1\}$
with $i\neq j$, we may choose $\wt F=(\wt F_1,\wt F_2,\ldots,\wt F_{2n+1})$ as above such that $(\wt F_i,\wt F_j)\colon M\to\c^2$ 
is a proper map. 
\end{theorem}

Theorem \ref{th:intro-Runge} shows in particular that every open Riemann surface properly embeds into $\c^3$ as a complex Legendrian curve. 
This result, which is analogous to Bryant's embedding theorem for compact Riemann surfaces as complex Legendrian curves in 
$\CP^3$ (see \cite[Theorem G]{Bryant1982JDG} and \cite{Segre1926}), 
has been a long-standing open problem in complex contact geometry.

Theorem \ref{th:intro-Runge} is a particular case of Theorem \ref{th:proper}, where the latter result
also ensures approximation of Mergelyan type on certain
{\em admissible subsets} (see Section \ref{sec:Mergelyan} for definitions and preliminary results). The key ingredients in the proof of 
Theorem \ref{th:proper} are a Mergelyan theorem for Legendrian curves (see Lemma \ref{lem:fixedcomponents}) and a general position
theorem ensuring that every holomorphic Legendrian curve $K\to\c^{2n+1}$, where $K$ is as in Theorem \ref{th:intro-Runge}, 
may be approximated in the $\Cscr^1(K)$-topology by holomorphic Legendrian embeddings $K\hra\c^{2n+1}$ (see Lemma 
\ref{lem:position}). The methods of proof exploit the classical Runge and Mergelyan approximation theorems for holomorphic functions on
open Riemann surfaces and the construction of period-dominating sprays of Legendrian curves. Further, to ensure the general position 
result, we use the classical proof of the transversality theorem due to Abraham \cite{Abraham1963TAMS}. 
Similar techniques have been developed by the authors in the theories of minimal surfaces in the real 
Euclidean space $\r^N$ $(N\ge3)$, null holomorphic curves in $\c^N$, 
and, more generally, holomorphic immersions of open Riemann surfaces into $\c^N$ which are directed by 
Oka conical subvarieties (see \cite{AlarconLopez2012JDG,AlarconLopez2013MA,AlarconLopez2014TAMS,AlarconForstneric2014IM,AlarconLopez2015GT,AlarconForstnericLopez2016MZ,AlarconForstnericLopez2016NoOrientable,AlarconForstnericLopez2016Gauss} and the references therein). 
The main difference here is that the holomorphic distribution controlling Legendrian curves depends on the base point, 
and this requires a novel approach. Finally, with Lemmas \ref{lem:fixedcomponents} and \ref{lem:position} in hand, 
Theorem \ref{th:proper} follows by a standard recursive argument.

Our second main theorem concerns the existence of complete bounded Legendrian curves with Jordan boundaries
in the contact manifold $(\C^{2n+1},\eta)$ (cf.\ \eqref{eq:contact}) for any $n\in \n$.

%
%
\begin{theorem}[Complete Legendrian curves with Jordan boundaries]\label{th:intro-Jordan}
Let $M$ be a compact bordered Riemann surface with nonempty boundary $bM$. Every Legendrian curve 
$F\colon M\to\c^{2n+1}$ $(n\in\n)$ of class $\Ascr^1(M)$ can be approximated uniformly on $M$ by continuous injective maps 
$\wt F\colon M\hra\c^{2n+1}$ whose restriction to the interior $\mathring M=M\setminus bM$ is a 
complete holomorphic Legendrian embedding $\mathring M\hra \c^{2n+1}$.
\end{theorem}

Recall that a {\em compact bordered Riemann surface} is the same thing as a smoothly bounded compact domain in an open Riemann surface
(see Section \ref{ss:RS} for a precise definition). An immersion $\varphi\colon R\to\c^N$ of a smooth open surface $R$ into $\c^N$ is 
said to be {\em complete} if the Riemannian metric on $R$ induced by the Euclidean metric of $\c^N$ via $\varphi$ is complete. 

The existence of complete bounded holomorphic Legendrian curves in $\c^{2n+1}$ for any $n\in\n$ is derived from 
Theorem \ref{th:intro-Jordan}; this settles the question posed by Mart\'in, Umehara, and Yamada in \cite[page 314]{MartinUmeharaYamada2014RMI}. Theorem \ref{th:intro-Jordan} is also connected to the problem, raised by 
Yang in 1977 (see \cite[Question II]{Yang1977} and \cite{Yang1977JDG}), about the existence of complete bounded 
immersed or embedded complex submanifolds of a complex Euclidean space.
For recent advances and a history of this problem, we refer to the papers by Alarc\'on and  Forstneri\v c \cite{AlarconForstneric2013MA}, 
Alarc\'on and L\'opez \cite{AlarconLopez2016JEMS}, Globevnik \cite{Globevnik2015AM}, and Alarc\'on et al. \cite{AlarconGlobevnikLopez2016Crelle}. 

The proof of Theorem \ref{th:intro-Jordan} requires, in addition to the above mentioned approximation and desingularization results
given by Lemmas \ref{lem:fixedcomponents} and \ref{lem:position}, 
to approximately solve certain Riemann-Hilbert type boundary value problems for Legendrian curves in $\c^{2n+1}$;
see Lemma \ref{lem:RH} and Theorem \ref{th:RH}.  

The analogues of Theorem \ref{th:intro-Jordan} have already been established for complex curves in $\c^k$ $(k\ge 2)$, 
minimal surfaces in $\r^N$ $(N\ge 3)$, and null holomorphic curves in $\c^N$ $(N\ge 3)$; 
see Alarc\'on, Drinovec Drnov\v sek, Forstneri\v c, and L\'opez \cite[Theorems 1.1 and 1.6]{AlarconDrinovecForstnericLopez2015PLMS}
and  also \cite{AlarconForstneric2015MA,AlarconDrinovecForstnericLopez2015} where approximate solutions to Riemann-Hilbert 
problems for minimal surfaces and null curves are provided. 
With the Mergelyan theorem, the desingularization theorem, and the Riemann-Hilbert method for Legendrian curves in hand, the proof of Theorem \ref{th:intro-Jordan} is an adaptation of \cite[proof of Theorem 1.1]{AlarconDrinovecForstnericLopez2015PLMS}.
For this reason, and with simplicity of exposition in mind, we provide the details only in the case when $M$ 
is the closed unit disk $\cd\subset\c$ (see Theorem \ref{th:Jordan}); the proof of the general case is a simple adaptation
of this special case as in the cited works. 
An important ingredient in the proof is the observation that almost every affine complex hyperplane of 
$\c^{2n+1}$ contains properly embedded Legendrian curves $\c\hra\c^{2n+1}$ passing through any given point in the 
hyperplane (see Proposition \ref{pro:flat}).

As a consequence of Theorem \ref{th:intro-Jordan} and Darboux's theorem for complex contact manifolds (see Theorem \ref{th:Darboux-contact}), we obtain the following existence result for complete, relatively compact, Legendrian curves in an arbitrary complex contact manifold.

\begin{corollary}\label{co:Wtau}
Let $(W,\Lscr)$ be a complex contact manifold. 
Given any  compact bordered Riemann surface $M$, there exists a continuous injective map $M\hra W$ whose restriction 
to $\mathring M$ is a holomorphic Legendrian embedding that is complete with respect to every Riemannian metric on $W$.
\end{corollary}
\begin{proof}
Let $\dim W=2n+1\ge 3$.
By Darboux's theorem (see Theorem \ref{th:Darboux-contact}), every point of $W$ has a neighborhood $U\subset W$ 
and holomorphic coordinates $\Phi=(x_1,y_1,\ldots,x_n,y_n,z)$ on $U$ such that $\Lscr|_U= \ker \Phi^*(\eta)$, where  $\eta$ is the 
standard contact form given by \eqref{eq:contact}. Let $V\Subset U$ be a relatively compact domain. Theorem \ref{th:intro-Jordan} 
provides a continuous injective map $F_0\colon M\to \Phi(V)\subset \c^{2n+1}$ whose restriction to $\mathring M$ is a holomorphic 
$\eta$-Legendrian  embedding which is complete with respect to the Euclidean metric $g_0:=|dz|^2+\sum_{j=1}^n (|dx_j|^2+|dy_j|^2)$
on $\c^{2n+1}$. As a consequence, $F:=\Phi^{-1}\circ F_0\colon M\to W$ is $\Lscr$-Legendrian. Further, since $\overline V$ is compact, 
the restriction to $V$ of any Riemannian metric $g$ on $W$ is comparable to $\Phi^*(g_0|_{\Phi(V)})$, and hence, 
$F|_{\mathring  M}$ is complete with respect to any such $g$. This completes the proof.
\end{proof}

The paper includes an Appendix in which we collect some results concerning holomorphic contact and symplectic forms and structures; 
in particular, the Darboux theorems. 
These results are well-known in the real case, but their complex analogues do not seem easily available in the literature. 
The proofs in the holomorphic case follow those for the smooth case rather closely,  and we do not claim any originality on them.

%
%
%

Our results open several natural new questions and possible directions of future investigation. 
Explicitly, we pose the following problems.

\begin{problem} Assume that $W$ is a complex manifold of dimension $n\ge 4$ and $\Lscr\subset TW$ is a completely nonintegrable 
holomorphic subbundle of dimension $m$ with $2\le m\le n-2$ (i.e., the repeated commutators 
of holomorphic vector fields tangent to $\Lscr$ span $TW$). Does Corollary \ref{co:Wtau} hold in this setting, i.e., 
does every bordered Riemann surface admit a bounded complete holomorphic map (immersion, embedding) 
to $W$ which is tangent to $\Lscr$?
\end{problem}

\begin{problem}
It has recently been shown by the second named author that for any $n\ge 3$ there exists a holomorphic contact structure on $\C^{2n+1}$ 
which is Kobayashi hyperbolic, and in particular is not globally contactomorphic to the standard one (see \cite{Forstneric2016}).
Are there infinitely many, or perhaps even uncountable many pairwise nonequivalent complex contact structures on $\C^3$? 
(Eliashberg showed that on $\R^3$ there exist countably many different isotopy classes of smooth contact structures \cite{Eliashberg1989IM,Eliashberg1993IMRN}.)
\end{problem}

\begin{problem}
Does the analogue of Theorem \ref{th:intro-Runge} hold for maps of bordered Riemann surfaces 
into an arbitrary Stein contact manifold $(W,\Lscr)$?
\end{problem}

An even more ambitious problem is to develop methods for constructing higher-dimensional complex Legendrian 
submanifolds in complex contact manifolds. 
(We refer to Landsberg and Manivel \cite{LandsbergManivel2007AJM} 
for examples of compact Legendrian submanifolds in the projective space $\CP^{2n+1}$ endowed with the standard contact structure.) 
One of the main questions in this direction is the following.

\begin{problem}
Let $X^n$ be an $n$-dimensional Stein manifold for some $n>1$. (Recall that $1$-dimensional Stein manifolds
are open Riemann surfaces.) Does $X$ admit a proper holomorphic contact map (immersion, embedding) into
to the standard complex contact manifold $(\C^{2n+1},\eta)$? 
\end{problem}


\section{Preliminaries} 
\label{sec:prelim}

%
%

\subsection{The standard holomorphic contact structure on $\c^{2n+1}$}\label{ss:eta}

Let $\eta$ denote the contact form \eqref{eq:contact} on the Euclidean space $\C^{2n+1}$ for some $n\in\N$.
Its differential
\[
	d\eta = \sum_{j=1}^n dx_j \wedge dy_j
\] 
is the standard {\em holomorphic symplectic form} on $\C^{2n}_{x_1,y_1,\ldots,x_n,y_n}$, and 
\[
	\eta\wedge(d\eta)^n=n! \, dx_1\wedge dy_1\wedge\cdots \wedge dx_n\wedge dy_n\wedge dz
\] 
is a multiple of the standard holomorphic volume form on $\C^{2n+1}$. 
Note that $(\C^{2n+1},\eta)$ is contactomorphic to the restriction o the holomorphic contact structure $\Lscr$ on the projective space 
$\CP^{2n+1}$ obtained by projectivizing the standard symplectic holomorphic structure $(\C^{2n+2},\alpha)$ given by the 
symplectic form $\alpha=\sum_{i=0}^n dx_i\wedge dy_i$. Explicitly, for every complex line $\C v\subset \C^{2n+2}$
representing a point $[v] \in \CP^{2n+1}$ we let $\Lscr_v=\{[w]\in \CP^{2n+1}: \langle \alpha,v\wedge w\rangle=0\}$.

Let us write $\frac{\di}{\di x_j}=\di_{x_j}$ and similarly for the other coordinates on $\C^{2n+1}$.
Note that $\Lscr= \ker\eta$ is a trivial bundle that is spanned at each point by the holomorphic vector fields
\begin{equation}\label{eq:spanning}
	\di_{x_j},\quad  \di_{y_j}-x_j \di_z,\quad j=1,\ldots,n.
\end{equation}
Furthermore, we have that
\begin{equation}\label{eq:transverse}
	\big\langle \eta, \di_z \big\rangle =1\quad \text{and}\quad \di_z \,\rfloor\, d\eta=0
\end{equation}
where $\rfloor$ denotes the interior product: 
\[
	\langle \di_z \,\rfloor\, d\eta,V\rangle = \langle d\eta,\di_z \wedge V \rangle
	\ \  \text{for any vector field $V$}.
\]
Hence, $\di_z$ is the {\em Reeb vector field} of the contact manifold $(\C^{2n+1},\eta)$ (cf.\ \eqref{eq:Reeb}).
On $\C^3$ with complex coordinates $x,y,z$ we have 
\[
	\eta=dz+xdy,\quad \eta\wedge d\eta = dx\wedge dy\wedge dz.
\]
The projection $(x,y,z)\mapsto (y,z)$ is called the 
{\em front projection} and $(x,y,z)\mapsto (x,y)$ is the {\em Lagrange projection}.

Note that the holomorphic distribution $\Lscr=\ker\eta\subset T\C^{2n+1}$ is completely noninvolutive.
In fact, we have that $[\di_{x_j},\di_{y_j}-x_j \di_z]=-\di_z$ and the vector fields
\eqref{eq:spanning} together with $\di_z$ clearly span $T\C^{2n+1}$.
It follows that the real and imaginary parts of these vector fields, along with their commutators,
span $T\C^{2n+1}$ over $\R$.

The following observation will be important at several points of our argumentation.

\begin{remark}\label{rem:etaprime}
For each $j\in \{1,\ldots,n\}$, the holomorphic automorphism of $\c^{2n+1}$ given by
\[
    \Phi_j(x_1,y_1,\ldots,x_n,y_n,z)=(x_1',y_1',\ldots,x_n',y_n',z'),
\]
where  $(x_j',y_j')=(x_j,-y_j)$, $(x_i',y_i')=(x_i,y_i)$ for all $i\neq j$, and $z'=z+x_j y_j$, 
is an involution satisfying 
\[
\Phi_j^*\eta= dz+y_jdx_j+\sum_{i\neq j} x_idy_i.
\]
More precisely, setting $\eta_j=dz+y_jdx_j+\sum_{i\neq j} x_idy_i$, we have $\Phi_j^*\eta =\eta_j$, and hence the contact manifolds $(\c^{2n+1},\ker \eta)$ and $(\c^{2n+1},\ker \eta_j)$ are contactomorphic and $\Phi_j$ is a contactomorphism between them. In particular, $\Phi_j$ maps $\eta$-Legendrian curves to $\eta_j$-Legendrian curves and vice versa. Thus, the role of the variables $x_j$ and $y_j$ can be interchanged in many arguments. 
\end{remark}

%
%

\subsection{Riemann surfaces and mapping spaces}\label{ss:RS}

For $n\in\n$ we denote by $|\cdot|$ the Euclidean norm in $\c^n$. 
Given a topological space $L$ and a map $f\colon L\to\c^n$ we denote by 
$\|f\|_{0,L}:=\sup\{|f(u)|\colon u\in L\}$ the supremum norm of $f$.

Let $M$ be an open Riemann surface.
Given a subset $K\subset M$, we denote by $\Oscr(K)$ the algebra of all holomorphic functions on 
open neighborhoods of $K$ in $M$, where we identify any pair of functions which agree on some neighborhood of $K$. 
In particular, $\Oscr(M)$ denotes the algebra of all holomorphic functions $M\to\c$. 

If $K$ is a smoothly bounded compact domain in $M$ and $r\in\z_+=\{0,1,2,...\}$, 
we denote by $\Cscr^r(K)$ the algebra of all $r$ times continuously differentiable 
complex-valued functions on $K$ and by $\Ascr^r(K)$ the subalgebra of $\Cscr^r(M)$
consisting of all functions that are holomorphic in the interior $\mathring K=K\setminus bK$.
We denote by $||f||_{r,K}$ the standard $\Cscr^r$ norm of a function $f\in\Cscr^r(K)$, where the derivatives
are measured with respect to a Riemannian metric on $M$; the precise choice of the metric
will not be important for our purposes. We shall use the same notation for maps 
$f=(f_1,\ldots,f_n):K\to\C^n$ with $f_j\in \Cscr^r(K)$ for $j=1,\ldots,n$.

A {\em compact bordered Riemann surface} is a compact Riemann surface $M$ with nonempty boundary $\emptyset\neq bM\subset M$ 
consisting of finitely many pairwise disjoint smooth Jordan curves. The interior $\mathring M=M\setminus bM$ of such $M$ is called a 
{\em bordered Riemann surface}. It is classical that every compact bordered Riemann surface $M$ is diffeomorphic to a smoothly 
bounded compact domain in an open Riemann surface $\wh M$, and so the function spaces $\Ascr^r(M)$, $r\in\z_+$, are defined as above.

%
%

\subsection{Sprays of holomorphic maps}\label{ss:sprays}

We shall frequently use the notion of a {\em holomorphic spray} of maps $X\to Y$
between a pair of complex manifolds. This is simply a holomorphic map
$F\colon X\times W\to Y$, where $W$ is a connected domain in a Euclidean space $\C^N$ containing the origin.
We often consider $F_w=F(\cdotp,w)\colon X\to Y$ as a family of holomorphic maps depending
holomorphically on the parameter $w\in W$. The map $F_0=F(\cdotp,0)$ is called the {\em core} of the spray.
The spray is said to be {\em dominating} at a point $x\in X$ if the map 
$w\mapsto F(x,w)\in Y$ has maximal rank equal to $\dim Y$ at $w=0$; 
if this holds at every point of $X$ then the spray is said to be {\em dominating} (on $X$).

Sprays are a useful tool in linearization problems. In particular, in this paper we use 
{\em period dominating sprays} in order to control periods of holomorphic maps from Riemann
surfaces in approximation problems; see \eqref{eq:derivative-period} where this 
notion is first introduced. For a more complete information on holomorphic 
sprays and their applications we refer the reader to \cite{Forstneric2011book}.


\section{The Riemann-Hilbert method for Legendrian curves} 
\label{sec:contact}

We shall write $\D=\{\zeta\in \C: |\zeta|<1\}$ and $\T=b\D=\{\zeta\in \C: |\zeta|=1\}$.

%
%
\begin{lemma}\label{lem:approximate}
For every holomorphic disk $F=(x,y,z) \colon\D\to\C^{2n+1}$ the map
\begin{equation}\label{eq:tildeF}
	\wt F(\zeta) = \left(x(\zeta),y(\zeta),z(\zeta)-\int_0^\zeta F^*\eta\right),\quad \zeta\in \D
\end{equation}
is a holomorphic Legendrian disk. In particular, for every holomorphic disk $F\colon \D\to \C^{2n+1}$  
there exists a Legendrian disk $\wt F\colon\D\to\C^{2n+1}$ satisfying 
\[
	||\wt F-F||_{0,\D}\le \sup_{|\zeta|<1} \left| \int_0^\zeta F^*\eta\right |. 
\]
\end{lemma}

\begin{proof}
We have 
$
	(\wt F^*\eta)(\zeta) = z'(\zeta)d\zeta  - (F^*\eta)(\zeta) + \sum_{j=1}^n x_j(\zeta)y'_j(\zeta)d\zeta=0.
$
\end{proof}

The following lemma provides approximate solutions to the Riemann-Hilbert problem for holomorphic Legendrian disks.

%
%
\begin{lemma}\label{lem:RH}
Assume that $f=(x,y,z)\colon \cd\to\C^{2n+1}$ is a holomorphic Legendrian disk of class $\Ascr^1(\d)$,
and for every $u\in \T$ the map 
\[
	\cd \ni v\longmapsto F(u,v)=\bigl(X(u,v),Y(u,v),Z(u,v)\bigr) \in \C^{2n+1}
\]
is a Legendrian disk of class $\Ascr^1(\d)$ depending continuously on $u\in \T$ and such that 
$F(u,0)=f(u)$ holds for all $u\in \T$. Given numbers $\epsilon>0$ and $0<\rho_0<1$,
there exist a number $\rho'\in [\rho_0,1)$ and a holomorphic Legendrian disk $G\colon\overline{\d}\to\C^{2n+1}$ 
such that $G(0)=f(0)$ and the following conditions hold:
\begin{enumerate}
\item[\rm (i)]   $\sup\{|G(u)-f(u)| : |u|\le \rho'\} <\epsilon$,
\vspace{1mm}
\item[\rm (ii)]  $\dist(G(u),F(u,\T))<\epsilon$ for all $u\in \t$, and
\vspace{1mm}
\item[\rm (iii)] $\dist(G(\rho u),F(u,\overline{\D}))<\epsilon$ for all $u\in \T$ and all 
$\rho\in [\rho',1)$.
\end{enumerate}
If in addition $I$ is a proper closed segment in the circle $\t$
and $F(u,v)=f(u)$ for all $u\in \T\setminus I$ and $v\in\cd$, then for every 
open neighborhood $U$ of $I$ in $\overline{\d}$ we may choose $G$ as above such 
that it also satisfies the following condition:
\begin{enumerate}
\item[\rm (iv)] $G$ is $\epsilon$-close to $f$ in the $\Cscr^1$ topology on $\overline{\d}\setminus  U$.
\end{enumerate}
\end{lemma}

\begin{proof}
For simplicity of notation we shall consider the case $n=1$; the same proof will apply also for $n>1$.

Since we are looking for Legendrian disks $G\colon \cd\to\C^3$ satisfying certain approximate conditions
in relationship to $f$ and $F$, we may assume by approximation that all our holomorphic Legendrian disks are defined on a fixed open neighborhood of $\cd$. 
Indeed, for $f$ just use Mergelyan's approximation  for $x$ and $y$ and define $z$ in a neighborhood of $\overline \d$ accordingly; for the boundary disks proceed likewise but using the parametric version of Megelyan's theorem. 
 For $(u,v)\in \T\times \cd$ we have
\vspace{1mm}
\begin{eqnarray*} 
	X(u,v) &=&  \sum_{j\ge 0} a_j(u) v^j, \\
	Y(u,v) &=&  \sum_{k\ge 0} b_k(u) v^k, \\
         Z(u,v) &=& \sum_{n\ge 0} c_n(u) v^n,
\end{eqnarray*}
where the coefficients $a_j$, $b_k$, and $c_n$ are continuous functions of $u\in\T$ and we have
\[
	a_0(u)=x(u),\ \ b_0(u)=y(u),\ \ c_0(u)=z(u).
\]
The Legendrian condition for the map $v\mapsto F(u,v)$ (with a fixed $u\in\T$)
says that $Z_v+X Y_v=0$, where the
subscript denotes the partial derivative with respect to the indicated variable. 
From the power series expansions for $X$ and $Y$ we obtain
\[
	X Y_v = \sum_{j\ge 0} a_j v^j  \,\, \cdotp  \sum_{k\ge 1} b_k k v^{k-1} = 
	\sum_{n\ge 1} \biggl(\, \sum_{j+k=n} k a_j b_k \biggr) v^{n-1}. 
\]	 
Comparison with 
\[
	Z_v = \sum_{n\ge 1} c_n n v^{n-1}
\]
gives the equations
\[
	c_n= - \frac{1}{n}  \sum_{j+k=n} ka_j b_k,\quad n=1,2,\ldots.
\]
By approximation we may assume that there are only finitely many nonzero coefficients 
$a_j$, $b_k$ and hence $c_n$, i.e., the Legendrian curves $v\mapsto F(u,v)$ are polynomial 
in $v\in\C$ of bounded degree independent of $u\in \T$.
Furthermore, we may approximate each of the coefficients $a_j$ and $b_k$ (which are continuous functions
on $\T$) by a rational function with the only pole at $0$. In view of the
above formulas for the coefficients $c_n$ of $Z$ these also become rational functions on $\C$
with the only pole at $0$. We denote the resulting functions and maps by the same letters.
Note that this gives a family of polynomial Legendrian curves 
$F(u,\cdotp)=\left(X(u,\cdotp), Y(u,\cdotp) , Z(u,\cdotp)\right)\colon \C\to\C^3$ for $u\in \C\setminus \{0\}$
which are Laurent polynomials in the variable $u$. In particular, we have that
\begin{equation}\label{eq:v-equation}
	Z_v(u,v)+X(u,v) Y_v(v,u)=0,\quad u\in \C\setminus \{0\},\ v\in \C.
\end{equation}
Let $N_0$ be the biggest degree of pole of any of these coefficients at $0$.
For any $N\in \N$ with $N> N_0$ the map $F_N\colon\C\to\C^3$ given by
\[
	F_N(u) = (X_N(u),Y_N(u),Z_N(u)) := F(u,u^N)= \bigl(X(u,u^N),Y(u,u^N),Z(u,u^N)\bigr)
\]
is a holomorphic polynomial with $F_N(0)=f(0)$. It is well-known that for sufficiently big $N\in\N$
the map $F_N$ satisfies properties (i)--(iv) in the lemma (see e.g.\ \cite[Lemma 3.1]{DrinovecForstneric2012IUMJ}).

Although the maps $F_N$ obtained in this way need not be Legendrian, we shall now show 
that for all sufficiently big $N\in\N$ the map $F_N$ is as close as desired uniformly on $\cd$
to a Legendrian disk $G=G_N=(X_N,Y_N,\wt Z_N) \colon \cd\to \C^3$ of the form given by
Lemma \ref{lem:approximate}; this will complete the proof.
To this end we estimate the expression 
\[
	F_N^*\eta=dZ_N + X_NdY_N.
\] 
We have that
\begin{eqnarray*}
	\frac{d}{du} Z(u,u^N) &=& Z_u(u,u^N) + Z_v(u,u^N)Nu^{N-1}, \\
	X(u,u^N)\frac{d}{du} Y(u,u^N) &=&  X(u,u^N) \left(Y_u(u,u^N)+Y_v(u,u^N) N u^{N-1}\right).
\end{eqnarray*}
By adding these two equations and taking into account the condition $Z_v+XY_v=0$ for the 
Legendrian disk $v\mapsto F(u,v)$ we obtain
\[
	F_N^*\eta/du = Z_u+Z_v Nu^{N-1} + XY_u + XY_v N u^{N-1} = Z_u  + XY_u.
\]
From the power series expansions of $F=(X,Y,Z)$ we get
\begin{eqnarray*}
	(Z_u  + XY_u)(u,u^N)  &=& \sum_{n\ge 0} c'_n(u) u^{nN} +
		\sum_{j\ge 0} a_j(u) u^{jN} \,\,\cdotp \sum_{k\ge 0} b'_k(u) u^{kN} \\
		&=& \sum_{n\ge 1} \biggl( c'_n(u) + \sum_{j+k=n}a_j(u) b'_k(u) \biggr) u^{nN}.
\end{eqnarray*}
The term with $n=0$ in the above sum equals $c'_0+a_0b'_0=z'+xy'=0$ since the
curve $f=(x,y,z)$ is Legendrian, and hence it drops out from the sum. 

Recall that each of the coefficients $a_j$, $b_k$ and $c_n$ is Laurent 
polynomial of the form $P(u,1/u)$ where $P$ is a holomorphic polynomial on $\C^2$. 
The same is then true for their derivatives, and hence 
for the coefficients $c'_n + \sum_{j+k=n}a_j b'_k$ in the above expansion of $Z_u  + XY_u$. 
Let $N_1$ be the maximal power of $1/u$ that appears in any of these finitely many coefficients.
If $N\ge N_1$ then the function $(Z_u  + XY_u)(u,u^N)$ is a polynomial in $u$ whose lowest order term is $u^{N-N_1}$
or higher. Note that $\int_0^\zeta u^{N-N_1}du = \zeta^{N-N_1+1}/(N-N_1+1)$ which converges
to zero uniformly on the disk $|\zeta|\le 1$ when $N\to\infty$. Since we have finitely
many such terms in the sum, it follows that the integral
\[
	\int_0^\zeta F_N^*\eta = \int_0^\zeta (Z_u  + XY_u)(u,u^N)\, du
\]
converges to zero uniformly on $\cd$ as $N\to+\infty$.
Furthermore, if the last assumption in the lemma holds 
then we can perform the same construction on a somewhat bigger compact simply connected domain 
$D\subset\C$ containing the disk $\cd$ such that $I\subset bD$ and $\cd\setminus U\subset\mathring D$; 
the $\Cscr^0$-estimate on $D\cong \cd$ then yields $\Cscr^1$ estimate on $\cd\setminus U$ in view of the Cauchy estimates.
Finally, setting $\wt Z_N(\zeta)= Z_N(\zeta) - \int_0^\zeta F_N^*\eta$ 
(cf.\ Lemma \ref{lem:approximate}) gives a sequence of holomorphic Legendrian maps 
$G_N=(X_N,Y_N,\wt Z_N) \colon \C\to \C^3$ satisfying the conclusion of Lemma \ref{lem:RH}.
\end{proof}

We now prove the analogous result for any bordered Riemann surface.

%
%
\begin{theorem}\label{th:RH}
Assume that $M$ is a compact bordered Riemann surface, 
$I\subset bM$ is an arc which is not a boundary component of $M$,
$f=(x,y,z)\colon M \to\C^{2n+1}$ is a Legendrian map of class $\Ascr^1(M)$, 
and for every point $u\in bM$ the map
\[
	\cd \ni v\longmapsto F(u,v)=\bigl(X(u,v),Y(u,v),Z(u,v)\bigr) \in \C^{2n+1}
\]
is a Legendrian disk of class $\Ascr^1(\d)$, depending continuously on $u\in bM$,  such that 
$F(u,0)=f(u)$ for all $u\in bM$ and $F(u,v)=f(u)$ for all $u\in bM\setminus I$ and $v\in\cd$. 
Given a number $\epsilon>0$ and a neighborhood $U\subset M$ of the arc $I$, there exist a holomorphic 
Legendrian map $H \colon M \to\C^{2n+1}$ and a neighborhood $V\Subset U$ of $I$ 
with a smooth retraction $\rho\colon V\to V\cap bM$ such that the following conditions hold:
\begin{enumerate}
\item[\rm (i)]   $\sup\{|H(u)-f(u)| : u\in M\setminus V\} <\epsilon$,
\vspace{1mm}
\item[\rm (ii)]  $\dist(H(u),F(u,\T))<\epsilon$ for all $u \in bM$, and
\vspace{1mm}
\item[\rm (iii)] $\dist(H(u),F(\rho(u),\overline{\D}))<\epsilon$ for all $u\in V$.
\end{enumerate}
\end{theorem}

\begin{proof}
For simplicity of notation we consider the case $n=1$; the same proof applies in general by considering $x$ and $y$ as 
vector-valued functions and writing $xdy=\sum_{i=1}^n x_i\, dy_i$. 

We may assume that $M$ is connected. 
Choose a closed smoothly bounded simply connected domain $D\subset U$ (a disk) 
such that $D$ is a neighborhood of the arc $I$. By denting $bM$ slightly inward along a neighborhood of 
$I$ we can find a smoothly bounded compact domain $M' \subset M$ 
such that $M=M'\cup D$ and the following separation condition holds:
\begin{equation}\label{eq:separation}
	\overline{M'\setminus D} \, \cap \, \overline {D\setminus M'}=\emptyset.
\end{equation}
Thus, $(M', D)$ is a {\em Cartan pair} (cf.\ \cite[Definition 5.7.1]{Forstneric2011book}).

Let $C_1,\ldots,C_\ell\subset \mathring M'$ be closed curves forming a basis of the homology group
$H_1(M';\Z)\cong H_1(M;\Z)=\Z^\ell$ such that the union $\bigcup_{j=1}^\ell C_j$ is Runge in $M$. 
Consider the {\em period map} 
\[
	\Pcal=(\Pcal_1,\ldots,\Pcal_\ell)  : \Ascr^1(M)^2\to\C^\ell
\]
whose $j$-th component equals
\begin{equation}\label{eq:periodmap}
	\Pcal_j(x,y)=\int_{C_j} x\, dy,\qquad x,y \in \Ascr^1(M).
\end{equation}
Note that $\Pcal(x,y)=0$ if and only if the 1-form $xdy$ is exact, and this holds if and only if
$(x,y)$ is the Lagrange projection of a Legendrian curve $f=(x,y,z)\colon M\to \C^3$. 

We shall first assume that the second component $y$ of $f$ 
is not constant; then $y|_{C_j}$ is not constant for any $j=1,\ldots,\ell$ by the identity principle.
By the Runge property of $\bigcup_{j=1}^\ell C_j$ there exist holomorphic functions $g_1,\ldots,g_\ell$
on $M$ such that for every $j,k=1,\ldots,\ell$ the number $\int_{C_j} g_k\, dy\approx \delta_{j,k}$ is 
close to $1$ if $j=k$ and to $0$ if $j\ne k$. (Here, $\delta_{j,k}$ is the Kronecker symbol.
We first construct smooth functions $g_k$ on $\bigcup_{j=1}^\ell C_j$ such that
$\int_{C_j} g_k\, dy =\delta_{j,k}$ and then use Mergelyan's theorem to approximate them 
by holomorphic functions on $M$.) Let $\zeta=(\zeta_1,\ldots,\zeta_\ell)\in\C^\ell$. 
Consider the function $\wt x\colon M\times \C^\ell\to\C$ given by 
\begin{equation}\label{eq:X-spray}
	\wt x(u,\zeta) = x(u) + \sum_{k=1}^\ell \zeta_k\, g_k(u),\quad u\in M,\ \zeta\in\C^\ell. 
\end{equation}
Note that for all $j,k\in\{1,\ldots,\ell\}$ we have 
\begin{equation}\label{eq:dizetak}
	\frac{\di}{\di\zeta_k}\bigg|_{\zeta=0} \int_{C_j} \wt x(\cdotp,\zeta)\, dy = \int_{C_j} g_k\, dy \, \approx\, \delta_{j,k}.
\end{equation}
If the above approximations are close enough then
\begin{equation}\label{eq:derivative-period}
	\frac{\di}{\di\zeta}\bigg|_{\zeta=0} \Pcal(\wt x(\cdotp,\zeta),y)  : \C^\ell \lra \C^\ell
	\ \ \text{is an isomorphism}.
\end{equation}
If this holds, then the map \eqref{eq:X-spray} is called a {\em period dominating holomorphic spray} with the core $\wt x(\cdotp,0)=x$.

Assume now that $(x,y)\in \Ascr^1(M)^2$ is the Lagrange projection of the given 
Legendrian map $f=(x,y,z)\colon M\to \C^3$; hence $\Pcal(x,y)=0$. By the inverse function theorem
there is a ball $r\B\subset \C^\ell$ around the origin such that the map 
$r\B\ni \zeta \mapsto \Pcal(\wt x(\cdotp,\zeta),y)\in \C^\ell$
is biholomorphic onto its image (a neighborhood of $0\in\C^\ell$).
Fix a point $u_0\in D$ and consider the function $\wt z \colon D\times \C^\ell\to \C$ 
of class $\Ascr^1(D \times \C^\ell)$ given by
\[
	\wt z(u,\zeta) = z(u_0) - \int_{u_0}^u \wt x(\cdotp,\zeta)\, dy,\quad u\in D,\ \zeta\in \C^\ell.
\]
Recall that $z$ is the third component of the Legendrian map $f=(x,y,z)\colon M\to\C^3$
and the integral is taken over any path in the disk $D$. Note that $\wt z(\cdotp,0)=z|_D$
since $\wt x(\cdotp,0)=x$ and $dz=-xdy$. Let $\wt f\colon D\times \C^\ell\to\C^3$ be the family of Legendrian disks
\begin{equation}\label{eq:tildef}
	D \ni u \mapsto \wt f(u,\zeta)=\bigl( \wt x(u,\zeta),y(u),\wt z(u,\zeta)\bigr) \in \C^3
\end{equation}
depending holomorphically on $\zeta\in\C^\ell$. (The second component $y$ is independent of
the parameter $\zeta$.) Note that $\wt f(u,0)=f(u)$ for $u\in D$.

For each point $u \in bD\cap bM$ and for every $\zeta\in\C^\ell$ we let
\[
	\cd  \ni v\mapsto \wt F(u,v,\zeta)= \bigl(\wt X(u,v,\zeta),\wt Y(u,v,\zeta),\wt Z(u,v,\zeta)\bigr) \in \C^{3}
\]
be the Legendrian disk of class $\Ascr^1(\cd)$ given by
\begin{eqnarray*}
	\wt X(u,v,\zeta) &=& X(u,v)+\wt x(u,\zeta)-x(u), \\
	\wt Y(u,v,\zeta) &=& Y(u,v), \\
	\wt Z(u,v,\zeta) &=& \wt z(u,\zeta) - \int_{t=0}^{t=v} \wt X(u,t,\zeta)\, d\wt Y(u,t). 
\end{eqnarray*}
When $\zeta=0$, we have $\wt X(u,v,0)=X(u,v)$, $\wt Y(u,v,0) = Y(u,v)$ and hence
\[
	\wt Z(u,v,0) = z(u)-\int_0^v X(u,\cdotp) dY(u,\cdotp)=Z(u,v), 
\]
so we see that $\wt F(u,v,0)=F(u,v)$ is the given Legendrian disk in the theorem. Furthermore, 
setting $v=0$ we have
\[
	\wt F(u,0,\zeta)=\wt f(u,\zeta),\quad u\in bD\cap bM,\ \zeta\in\C^\ell.
\]
Finally, for every point $u\in bD\cap bM\setminus I$ and for all $\zeta\in\C^\ell$ we have 
\[
	\wt F(u,v,\zeta)=\wt F(u,0,\zeta)=\wt f(u,\zeta),\quad v\in\cd,
\] 
so $\wt F(u,\cdotp,\zeta)$ is the constant disk. We extend $\wt F$ to all points $u\in bD$ by setting 
\[
	\text{$\wt F(u,v,\zeta)=\wt f(u,\zeta)$ for all $u\in bD\setminus I$, $v\in \cd$  and $\zeta\in\C^\ell$.}
\]
Note that $\wt f(\cdotp,\zeta)\colon D\to \C^3$ and $\wt F(u,\cdotp,\zeta)\colon \cd \to\C^3$ are families of holomorphic Legendrian maps, 
depending continuously on $u\in bD$ (this only pertains to $\wt F$) and holomorphically on $\zeta\in\C^\ell$, 
which satisfy the assumptions of Lemma \ref{lem:RH} on the disk $D\cong \cd$. 
Hence there is a family of Legendrian disks $\wt G(\cdotp,\zeta)\colon D\to\C^3$ 
satisfying the approximation conditions in Lemma \ref{lem:RH} with respect to the central Legendrian disk $\wt f(\cdotp,\zeta)$ 
and the family of Legendrian disks $\wt F(u,\cdotp,\zeta)$ over the boundary point $u\in bD$.
It is easily seen from the proof of Lemma \ref{lem:RH} that the family $\wt G(\cdotp,\zeta)$ may be chosen to
depend holomorphically on $\zeta\in \C^\ell$, and the estimates in the lemma can be made uniform
for all points $\zeta$ in any given bounded domain in $\C^\ell$, in particular, on the ball $r\B\subset\C^\ell$.

Let $V\subset D\setminus M'$ be a small neighborhood of the arc $I\subset bD$.
By condition (iv) in Lemma  \ref{lem:RH} we may assume that $\wt G(\cdotp,\zeta)$ is as close as desired to 
$\wt f(\cdotp,\zeta)$ in the $\Cscr^1$ norm on the set $D\setminus V$, and hence on $M'\cap D\subset D\setminus V$. 
In particular, given $\delta >0$ we may assume that
\[
	||\wt G(\cdotp,\zeta)-\wt f(\cdotp,\zeta)||_{1,M'\cap D} <\delta, \quad \zeta\in r\B.
\]
We shall write $\wt G=(\wt G_1,\wt G_2,\wt G_3)$. 

Recall that the first component of $\wt f$ (see \eqref{eq:tildef}) is the function $\wt x$ (cf.\ \eqref{eq:X-spray}) 
which is defined on all of $M$. By solving a Cousin-I problem with bounds on the Cartan pair $(M',D)$ we glue 
$\wt x$ and $\wt G_1$ and obtain a function $H_1 (\cdotp,\zeta)\colon M\to \C$ of class $\Ascr^1(M)$, 
holomorphic in $\zeta$, such that for all $\zeta\in r\B$ we have
\[
	||H_1(\cdotp,\zeta)-\wt x(\cdotp,\zeta)||_{1,M'} < C \delta,\quad  
	||H_1(\cdotp,\zeta)-\wt G_1(\cdotp,\zeta)||_{1,D} < C \delta,
\]
where the constant $C$ only depends on the Cartan pair $(M',D)$. 
This is accomplished by first patching the two functions
smoothly, using the separation condition \eqref{eq:separation}, and then correcting the $\Cscr^1$-small error by 
solving the $\dibar$-equation with $\Cscr^1$-estimates on the bordered Riemann surface $M$. The variable
$\zeta$ is treated as a parameter. This is standard: see e.g.\ 
\cite[proof of Lemma 8.5.2]{Forstneric2011book} and use $\Cscr^1$ estimates instead of $\Cscr^0$ estimates.

Likewise, we can glue the second component $\wt y=y\in \Ascr^1(M)$  of $\wt f$
with the function $\wt G_2(\cdotp,\zeta)$ into a function $H_2(\cdotp,\zeta)\colon M\to \C$ of class $\Ascr^1(M)$, 
holomorphic in $\zeta$, such that for all $\zeta\in r\B$ we have the estimates
\[
	||H_2(\cdotp,\zeta)- y||_{1,M'} < C  \delta,\quad  
	||H_2(\cdotp,\zeta)-\wt G_2(\cdotp,\zeta)||_{1,D} < C  \delta.
\]

From the above estimates on $M'$ and the fact that  $\bigcup_{j=1}^\ell C_j\subset M'$ 
it follows that the period map $\zeta \mapsto \Pcal(H_1(\cdotp,\zeta), H_2(\cdotp,\zeta))$ 
(cf.\ \eqref{eq:periodmap}) approximates the biholomorphic 
period map $\zeta\mapsto \Pcal(\wt x(\cdotp,\zeta), \wt y(\cdotp,\zeta))$
uniformly on the ball $\zeta\in r\B$.  Assuming that $\delta>0$ is chosen small enough, 
it follows that there is a point $\zeta'\in r\B$ near the origin such that 
\begin{equation}\label{eq:periodzero}
	\Pcal\bigl(H_1(\cdotp,\zeta'), H_2(\cdotp,\zeta')\bigr)=0.
\end{equation}
For this value of $\zeta'$ we obtain a Legendrian curve 
\[
	H=\bigl(H_1(\cdotp,\zeta'),H_2(\cdotp,\zeta'),H_3) : M\to \C^3
\]
whose third component equals
\[
	H_3(u) = z(u_0) - \int_{u_0}^u H_1(\cdotp,\zeta')\, dH_2(\cdotp,\zeta'),\quad u\in M.
\]
The integral is independent of the choice of the path in $M$ since by \eqref{eq:periodzero} all the periods 
over closed curves in $M$ vanish. (Note however that we do not get a Legendrian curve for 
parameter values $\zeta\ne \zeta'$ since the period condition \eqref{eq:periodzero} fails.) 
It follows from the construction that $H$ satisfies the conclusion of 
Theorem \ref{th:RH} provided that the approximations made in the proof were close enough.
This completes the proof under the assumption that the second component 
$y$ of $f$ is nonconstant.

Assume now that $y=y_0$ is constant. If the first component $x$ is nonconstant, we consider a  spray of
the form \eqref{eq:X-spray} over the second component:
\[
	\wt y(p,\zeta) = y_0 + \sum_{k=1}^\ell \zeta_k\, g_k(p),\quad p\in M,\ \zeta=(\zeta_1,\ldots,\zeta_\ell) \in\C^\ell,
\]
where the functions $g_1,\ldots, g_\ell\in \Oscr(M)$ are chosen such that 
\[
	\int_{C_j} x \, dg_k = - \int_{C_j} g_k \,  dx \, \approx\,  \delta_{j,k}. 
\]
This ensures that the period map $\zeta\mapsto \Pcal(x,\wt y(\cdotp,\zeta))$ has maximal rank at $\zeta=0$,
so  we can proceed as before, keeping the component $x$ fixed during the proof. 

Finally, if $x=x_0$ and $y=y_0$ are both constant, then any perturbation of either $x$ or $y$ integrates 
to a Legendrian curve which  brings us back to the second case considered above. 
Of course we must also adjust the Legendrian disks $F(u,\cdotp)$ $(u\in bM)$ accordingly so that the 
condition $F(u,0)=f(u)$ is satisfied.
\end{proof}


\section{Mergelyan approximation by embedded Legendrian curves} \label{sec:Mergelyan}

In this section we prove an approximation result of Runge-Mergelyan type for Legendrian curves by holomorphic 
Legendrian embeddings; see Lemma \ref{lem:position} below. 
This is an important step in the proof of Theorem \ref{th:proper} given in the following section.

Recall that a compact set $K$ in a complex manifold $M$ is said to be 
{\em $\Oscr(M)$-convex}, or {\em holomorphically convex}, or {\em Runge} in $M$, if for every point
$p\in M\setminus K$ there exists  $f\in \Oscr(M)$ with $|f(p)|>\max_K |f|$.  
If $M$ is an open Riemann surface, then a compact subset $K\subset M$ is Runge if and only if 
$M\setminus K$ has no relatively compact connected components in $M$.

\begin{definition}
\label{def:admissible}
A compact subset $S$ of an open Riemann surface $M$ is said to be {\em admissible} if $S=K\cup \Gamma$,
where $K=\bigcup \overline D_j$ is a union of finitely many pairwise disjoint, compact, smoothly bounded domains $\overline D_j$ in $M$ and $\Gamma=\bigcup \Gamma_i$ is a union of finitely many pairwise disjoint smooth arcs or closed curves that intersect $K$ only in their endpoints (or not at all), and such that their intersections with the boundary $bK$ are transverse. 
Note that $\mathring S=\mathring K$.
\end{definition}

Given an admissible set $S=K\cup\Gamma\subset M$, we shall use the notation 
\begin{equation}\label{eq:ArS}
		\Ascr^r(S)=\{f\in \Cscr^r(S) : f|_{\mathring K}\in \Oscr(\mathring K)\},\quad  r\in\z_+.
\end{equation}
The natural topology on $\Ascr^r(S)\subset \Cscr^r(S)$ coincides with the $\Cscr^r(K)$ topology on the subset $K$, while on each of 
the arcs $\Gamma_i\subset \Gamma$ we use the $\Cscr^r$-norm of the function measured with respect to a fixed regular
parametrization of $\Gamma_i$.
Note that an admissible set $S$ is Runge in $M$ if and only if the inclusion map $S\hra M$ induces an injective homomorphism 
$H_1(S;\z)\hra H_1(M;\z)$ of the first homology groups. If that is the case, the classical Mergelyan approximation theorem 
(see \cite{Mergelyan1951DAN}) ensures that every function $f\in \Ascr^r(S)$ $(r\in\z_+)$ 
can be approximated in the $\Cscr^r(S)$-topology by functions holomorphic on $M$. 
When there is no place for ambiguity, we will simply write $\Ascr^r(S)$ for $\Ascr^r(S)^n$, $n\in\n$.
 
Let $f\colon S\to\c$ be a function of class $\Ascr^1(S)$. Fix a holomorphic $1$-form $\theta$ vanishing nowhere on  $M$ and 
consider the continuous map $\wh f\colon S\to\c$ given by
\begin{itemize}
\item $\wh f=df/\theta$ on $\mathring K$; 
\vspace{1mm}
\item $\wh f(\alpha(t))=(f\circ\alpha)'(t)/\theta(\alpha(t),\dot\alpha(t))$
for any smooth regular path $\alpha$ in $M$ parametrizing a connected component $\Gamma_i$ of $\Gamma$.
\end{itemize}
Clearly, $\wh f$ is a well-defined map of class $\Ascr^0(S)$. By definition, we set
\begin{equation}\label{eq:df}
     df:=\wh f\theta,\quad f\in\Ascr^1(S).
\end{equation}
Obviously, $\wh f$ depends on the choice of $\theta$ but $df$ does not. If $f\in\Oscr(S)$ then $df$ \eqref{eq:df} agrees with 
the restriction of the exterior differential of $f$ to the points of $S$, i.e. $d(f|_S)=(df)|_S$.
Conversely, every  pair $(\theta,\wh f)$, where $\theta$ is a holomorphic $1$-form vanishing nowhere on  $M$ and $\wh f\colon S\to\c$ 
is a function of class $\Ascr^0(S)$ such that $\int_\gamma \wh f\theta=0$ for all closed curves $\gamma\subset S$, determines a 
function $f\colon S\to\c$ of class $\Ascr^1(S)$, with $df=\wh f\theta$, by the formula
\[
     f(p)=\int^p \wh f\theta,\quad p\in S.
\]

\begin{definition}
\label{def:generalized}
Let $S=K\cup \Gamma$ be a compact admissible set in an open Riemann surface $M$.
A map $f=(x_1,y_1,\ldots,x_n,y_n,z)\colon S\to\c^{2n+1}$  of class $\Ascr^1(S)$ 
is a {\em generalized Legendrian curve} if 
$f^*\eta=0$, where $\eta$ is the standard contact form \eqref{eq:contact}; that is to say, if
\[
       dz+\sum_{j=1}^n x_jdy_j=0\quad \text{everywhere on $S$}. 
\]
\end{definition}

As a preliminary step in the proof of Lemma \ref{lem:position}, we show the following approximation result which 
will be very useful in subsequent applications.

%
%
\begin{lemma}\label{lem:fixedcomponents}
Let $S=K\cup \Gamma$ be an admissible subset in an open connected Riemann surface $R$ (see Definition \ref{def:admissible}) such that 
$S$ is a deformation retract of $R$, and let $f=(x_1,y_1,\ldots,x_n,y_n,z)\colon S\to \C^{2n+1}$ be a generalized Legendrian curve 
(see Definition \ref{def:generalized}). Then $f$ may be approximated in the $\Cscr^1(S)$-topology 
by holomorphic Legendrian curves $\wt f\colon R\to \c^{2n+1}$ such that $\wt f$ has no constant component function.

Furthermore, assume that for some $\sigma \in\{1,\ldots,2n+1\}$ the following hold.
\begin{enumerate}[\rm (i)]
\item The $\sigma$-th component of $f$ is nonconstant and holomorphic on $R$.
\vspace{1mm}
\item If $\sigma=2n+1$ then there is $i\in\{1,\ldots,n\}$ such that $x_i$ and $y_i$ are not constant on any component of $K$, 
$x_i$ has no zeros in $\Gamma$, and $y_i$ has no critical points in $\Gamma$.
\end{enumerate} 
Then the approximating Legendrian curves $\wt f\colon R\to\c^{2n+1}$ can be chosen such that the $\sigma$-th component 
of $\wt f$ agrees with the  $\sigma$-th component of $f$.
\end{lemma}

\begin{proof} 
Since $S$ is a deformation retract of $R$, we have that $R$ is of finite topological type and
$S$ is connected and Runge in $R$. Without loss of generality, we can assume that $K\neq \emptyset$. Indeed, otherwise $S=\Gamma$ consists of a single closed curve or Jordan arc. Then we choose a small smoothly bounded close disk $K$ in $R$ such that $S'=K\cup \Gamma$ is admissible and connected, and $K\cap S$ is a single Jordan arc. Approximating $f$ by a generalized Legendrian curve $S'\to \c^{2n+1}$ reduces the proof to the case when $K\neq \emptyset$.

Assume as we may that $K\neq \emptyset$. It follows that every component of $\Gamma$ (hence of $S$) meets $K$.  Let $C_1,\ldots,C_\ell\subset S$ be closed curves forming a basis of the homology group $H_1(S;\Z)=\Z^\ell$ $(\ell\in\z_+)$ such that the union $\bigcup_{k=1}^\ell C_k$ is Runge  in $R$. By our assumptions on $S$, we may ensure that each curve $C_k$, $k=1,\ldots,\ell$, contains a subarc $\wt C_k$ lying in $\mathring K$.

The argument at the end of the proof of Theorem \ref{th:RH} and Remark \ref{rem:etaprime} enable us to assume without loss of generality that $y_1$ 
is not constant on any component of $K$. Let 
\[
	\Pcal=(\Pcal_1,\ldots,\Pcal_\ell)  : \Ascr^1(S)^{2n}\to\C^\ell 
\] 
be the period map \eqref{eq:periodmap} with the components 
\[
    \Pcal_k(g_1,h_1,\ldots,g_n,h_n)=\int_{C_k} \sum_{j=1}^n g_j\,dh_j,\quad k=1,\ldots,\ell. 
\]
Let us construct a spray $\wt x_1(\cdotp,\zeta)\colon S\to \C$ of the form
\begin{equation}\label{eq:X-spray-comp}
	\wt x_1(u,\zeta) = x_1(u) + \sum_{k=1}^\ell \zeta_k\, g_k(u),\quad u\in S,\ \zeta\in\C^\ell, 
\end{equation}
(cf.\ \eqref{eq:X-spray}) with the core $\wt x_1(\cdotp,0)=x_1$ (the first component of $f$), depending holomorphically on $\zeta\in \C^\ell$,
such that the $\zeta$-derivative of the period map $\Pcal(\wt x_1(\cdotp,\zeta),y_1,\ldots,x_n,y_n)$ at $\zeta=0$ is an isomorphism 
(cf.\ \eqref{eq:derivative-period}). We proceed as follows. Since $y_1$ is holomorphic and nonconstant  on each component of $K$, 
we may first construct smooth functions $g_k$ on $C$, with support on $\wt C_k\subset \mathring K$, defining the spray 
$\wt x_1(\cdotp,\zeta)$ on $C$ and then, by Mergelyan theorem, assume that each $g_k$ is holomorphic on $R$ 
(recall that $C$ is Runge in $R$). This ensures the existence of the desired spray on  $S$, the domain of definition of $x_1$.

Next, we approximate $(x_1,y_1,\ldots,x_n,y_n)\in\Ascr^1(S)^{2n}$ in the $\Cscr^1(S)$-norm
by a holomorphic map $(x_1',y_1',\ldots,x_n',y_n')\in\Oscr(R)^{2n}$ such that $\sum_{j=1}^nx_j'dy_j'$ does not vanish everywhere on $R$
and all the component functions $x_1',y_1',\ldots,x_n',y_n'$ are nonconstant. For that, recall that $S$ is Runge in $R$ and apply 
Mergelyan approximation with jet-interpolation. Let $\wt x_1'\colon R\times\c^\ell\to\c$ be the spray \eqref{eq:X-spray-comp}
obtained by replacing the core $x_1$ by $x_1'$:
\begin{equation}\label{eq:X'-spray-comp}
	\wt x_1'(u,\zeta) = x_1'(u) + \sum_{k=1}^\ell \zeta_k\, g_k(u),\quad u\in R,\ \zeta\in\C^\ell. 
\end{equation}
If the approximations are close enough, then the period map $\Pcal(\wt x_1'(\cdotp,\zeta),y_1',\ldots,x_n',y_n')$ 
is so close to $\Pcal(\wt x_1(\cdotp,\zeta),y_1,\ldots,x_n,y_n)$ that there is a point $\zeta'\in\C^\ell$ close to $0$
for which 
\[
	\Pcal(\wt x_1'(\cdotp,\zeta'),y_1',\ldots,x_n',y_n')=0.
\]
This means that $\wt x_1'(\cdotp,\zeta') dy_1'+\sum_{j=2}^nx_j'dy_j'$ is an
exact holomorphic $1$-form on $R$. Pick an initial point $u_0\in \mathring K$ and define 
the holomorphic function $z'\in \Oscr(R)$ by 
\[
	z'(u) = z(u_0) - \int_{u_0}^u \Big(\wt x_1'(\cdotp,\zeta') dy_1'+\sum_{j=2}^nx_j'dy_j'\Big),\quad u\in R.
\]
Clearly, the map $\wt f=\bigl(\wt x_1'(\cdotp,\zeta'),y_1',\ldots,x_n',y_n',z'\bigr) \colon R\to \C^{2n+1}$ is then a holomorphic Legendrian 
curve approximating $f$ in the $\Cscr^1(S)$-topology and having no constant component function provided that $\zeta'$ is close enough 
to $0\in\c^\ell$. This proves the first part of the lemma.

For the second part, let $\sigma\in\{1,\ldots,2n+1\}$ and assume that conditions {\rm (i)} and {\rm (ii)} hold. By Remark \ref{rem:etaprime} we may assume that $\sigma\in\{2,2n+1\}$.

\noindent{\em Case 1: Assume that $\sigma=2$.} In this case, the argument above works with the only difference that when 
approximating $(x_1,y_1,\ldots,x_n,y_n)\in\Ascr^1(S)^{2n}$ in the $\Cscr^1(S)$-norm
by a holomorphic map $(x_1',y_1',\ldots,x_n',y_n')\in\Oscr(R)^{2n}$, we choose $y_1'=y_1$. Take into account that, in this case, 
$y_1$ is holomorphic and nonconstant on $R$ by assumption {\rm (i)}, and so neither the argument at the end of the proof of 
Theorem \ref{th:RH} nor Remark \ref{rem:etaprime} are required.

\noindent{\em Case 2: Assume that $\sigma=2n+1$.} By {\rm (ii)} and Remark \ref{rem:etaprime} we may assume that $x_1$ and $y_1$ 
are not constant on any component of $K$, $x_1$ has no zeros in $\Gamma$, and $y_1$ has no critical points in $\Gamma$. 
Denote by $\mho$ the subset of $\Ascr^1(S)^{2n}$ consisting of those maps $(a_1,a_2,b_2,\ldots,a_n,b_n,w)\colon S\to\c^{2n}$ 
(note that $b_1$ is omitted) such that, for some (and hence for any) holomorphic $1$-form $\theta$ vanishing nowhere on $R$, the map
\[
      \frac1{a_1\theta}\big(dw+\sum_{j=2}^n a_jdb_j) : S\to\c
\]
is well-defined and of class $\Ascr^0(S)$ (see \eqref{eq:ArS} and \eqref{eq:df}). The fact that $f$ is a generalized Legendrian curve 
implies that $(x_1,x_2,y_2,\ldots,x_n,y_n,z)\in\mho$. Consider the period map
\[
	\Pcal=(\Pcal_1,\ldots,\Pcal_\ell)  : \mho\to\C^\ell
\] 
with the components 
\[
    \Pcal_k(a_1,a_2,b_2,\ldots,a_n,b_n,w)=\int_{C_k}  \frac1{a_1}\big(dw+\sum_{j=2}^n a_jdb_j),\quad k=1,\ldots,\ell.
\]
We shall now construct a spray $\wt x_1(\cdotp,\zeta)\colon S\to \C$ of the form
\begin{equation}\label{eq:X-spray-comp21}
	\wt x_1(u,\zeta) = x_1(u) e^{\sum_{l=1}^\ell \zeta_l\, g_l(u)},\quad u\in S,\ \zeta\in\C^\ell, 
\end{equation} 
where $g_l\in\Oscr(R)$ for all $l\in\{1,\ldots,\ell\}$,
with the core $\wt x_1(\cdotp,0)=x_1$, depending holomorphically on $\zeta\in \C^\ell$, such that  the $\zeta$-derivative of 
the period map $\Pcal(\wt x_1(\cdotp,\zeta),x_2,y_2,\ldots,x_n,y_n,z)$ at $\zeta=0$ is an isomorphism. 
(Note that $(\wt x_1(\cdotp,\zeta),x_2,y_2,\ldots,x_n,y_n,z)\in \mho$ for all $\zeta\in\c^\ell$ since $(x_1,x_2,y_2,\ldots,x_n,y_n,z)\in\mho$ 
and $e^{\sum_{l=1}^\ell \zeta_l g_l}$ has no zeros.) We proceed as follows. Note that, for such a spray,
\[
     	\frac{\di}{\di\zeta_l}\bigg|_{\zeta=0} \int_{C_k} 
	\frac1{\wt x_1(\cdot,\zeta)}\big(dz+\sum_{j=2}^n x_jdy_j)
	=-\int_{C_k} \frac{g_l}{x_1}\big(dz+\sum_{j=2}^n x_jdy_j) 
	 = \int_{C_k} g_l\, dy_1. 
\]

As above, since $\wt C_k\subset\mathring K$ and $y_1$ is holomorphic and nonconstant on every component of $K$, we may first 
construct smooth functions $g_l$ on $C$, with support on $\wt C_k$, defining the spray $\wt x_1(\cdotp,\zeta)$ on $C$, such that
\[
      \int_{C_k} g_l\, dy_1 \, \approx\, \delta_{l,k}.
\]
(Recall that $\delta_{l,k}$ is the Kronecker symbol.)
By the Mergelyan theorem, we may assume that each $g_l$ is holomorphic on $R$. As above, 
this guarantees the existence of the desired spray $\wt x_1(\cdotp,\zeta)$ on $S$.
We next approximate $(x_1,x_2,y_2,\ldots,x_n,y_n)\in\Ascr^1(S)^{2n-1}$ in the $\Cscr^1(S)$-norm
by a holomorphic map $(x_1',x_2',y_2',\ldots,x_n',y_n')\in\Oscr(R)^{2n-1}$ such that the following hold:
\begin{enumerate}[\rm (a)]
\item Each component function $x_1',x_2',y_2',\ldots,x_n',y_n'$ is nonconstant.
\vspace{1mm}
\item The zeros of $x_1'$ in $R$ are those of $x_1$ in $S$ (which lie in $K$, see assumption {\rm (ii)}), with the same order; in particular, $x_1'$ 
has no zeros in $R\setminus K$.
\vspace{1mm}
\item The $1$-form $dz+\sum_{j=2}^n x_j'dy_j'$ does not vanish identically on $R$, but it does vanish at the zeros of 
$dz+\sum_{j=2}^n x_jdy_j=-x_1dy_1$ in $K$, with the same order. 
(Observe that, by {\rm (ii)}, $x_1dy_1$ does not vanish identically on $K$ and has no zeros in $\Gamma$.)
\end{enumerate}
To ensure these conditions we make use of Mergelyan approximation with jet interpolation.

Let $\wt x_1'\colon R\times\c^\ell\to\c$ be the spray \eqref{eq:X-spray-comp21}
obtained by replacing the core $x_1$ by $x_1'$:
\begin{equation}\label{eq:X'-spray-comp21}
	\wt x_1'(u,\zeta) = x_1'(u) e^{\sum_{l=1}^\ell \zeta_l\, g_l(u)},\quad u\in R,\ \zeta\in\C^\ell. 
\end{equation}
Since $(x_1,x_2,y_2,\ldots,x_n,y_n,z)\in\mho$, conditions {\rm (b)} and {\rm (c)} above guarantee that 
\begin{equation}\label{eq:mho1}
\frac1{\wt x_1'(\cdot,\zeta)}\big(dz+\sum_{j=2}^n x_j'dy_j')\quad \text{is holomorphic on $R$ for all $\zeta\in\c^\ell$}.
\end{equation}
As above, if the approximations are close enough, there is a point $\zeta'\in\C^\ell$ close to $0$ for which 
$\Pcal(\wt x_1'(\cdotp,\zeta'),x_2',y_2',\ldots,x_n',y_n',z)=0$. 
Thus, choosing an initial point $u_0\in \mathring K$ and defining 
\[
	y_1'(u) = y_1(u_0) - \int_{u_0}^u \frac1{\wt x_1'(\cdot,\zeta')}\big(dz+\sum_{j=2}^n x_j'dy_j'),\quad u\in R,
\]
it follows from \eqref{eq:mho1} that $y_1'\colon R\to\c$ is holomorphic and, in view of {\rm (a)} and {\rm (c)}, the map 
\[
	\wt f=\bigl(\wt x_1'(\cdotp,\zeta'),y_1',\ldots,x_n',y_n',z\bigr)  : R\to \C^{2n+1}
\]
is a holomorphic Legendrian curve satisfying the conclusion of the lemma, provided that the approximations are sufficiently close.
\end{proof}

%
%

\begin{lemma}\label{lem:position}
Let  $M$ be a compact bordered Riemann surface. Every  Legendrian curve $f\colon M\to \C^{2n+1}$ $(n\in \n)$ of class $\Ascr^1(M)$  
may be approximated in the $\Cscr^1(M)$-topology by Legendrian embeddings $\wt f \colon M\hra \C^{2n+1}$ of class $\Ascr^1(M)$ 
having no constant component function.
\end{lemma}

\begin{proof} 
We assume that $n=1$; the same proof applies in general. We may also assume without loss of generality that $M$ is connected.

Let $R$ be an open Riemann surface containing $M$ as a smoothly bounded compact domain which is a deformation retract of $R$. 
By Lemma \ref{lem:fixedcomponents} we may assume that $f$ is a holomorphic Legendrian curve on $R$ having no constant 
component function. Write  $f=(x,y,z)\colon R\to \C^3$. 

The proof will proceed in two steps. In the first step we shall approximate $f$ in the $\Cscr^1(M)$-topology by a 
Legendrian immersion  $R\to\C^3$  whose $(x,y)$-projection 
is an immersion $R\to \C^2$. In the second step  we shall show how to remove double points on $M$ of the new Legendrian immersion  
and hence obtain a Legendrian embedding $M\hra\C^3$. If the approximation in both steps is close enough then the resulting 
Legendrian embedding will have no constant component function.

Let $C_1,\ldots,C_\ell\subset \mathring M$ be closed curves forming a basis of the homology group
$H_1(M;\Z)=\Z^\ell$ such that the union $\bigcup_{k=1}^\ell C_k$ is Runge  in $\mathring M$, hence in $R$.  Let 
\[
	\Pcal=(\Pcal_1,\ldots,\Pcal_\ell)  : \Ascr^1(M)^2\to\C^\ell
\] 
be the period map \eqref{eq:periodmap} with the components 
\[
    \Pcal_k(g,h)=\int_{C_k} g\,dh,\quad g,\,h\in \Ascr^1(M),\quad k=1,\ldots,\ell.
\]

For the first step of the proof, let $u_1,u_2,\ldots$ be the zeros of $dy$ in $R$. We shall approximate 
$x$ in the $\Cscr^1(M)$-norm by a function $x'\in \Oscr(R)$ such that 
\begin{equation}\label{eq:x'-conditions}
	\text{$dx'(u_j)\ne 0$\ \ for\ \ $j=1,2,\ldots$ \ \ and\ \ $\Pcal(x',y)=0$.}
\end{equation}
The first condition ensures that the map $(x',y)\colon R\to\C^2$ is an immersion, 
and the second one that $x'dy$ is an exact $1$-form on $R$.
Choosing an initial point $u_0\in M$ and setting 
\begin{equation} \label{eq:zprime}
	z'(u) = z(u_0) - \int_{u_0}^u x' dy,\quad u\in R
\end{equation} 
we shall obtain a desired Legendrian immersion $f'=(x',y,z')\colon R\to \C^3$.

We now explain how to find the function $x'\in \Oscr(R)$ satisfying \eqref{eq:x'-conditions}.
Choose a function $h\in \Oscr(R)$ such that for every $j=1,2,\ldots$ we have 
$dh(u_j)=0$ if $dx(u_j)\ne 0$ and $dh(u_j)\ne 0$ if $dx(u_j) = 0$; the existence of such a function is ensured, for instance, by the 
theorem of Gunning and Narasimhan (see \cite{GunningNarasimhan1967MA}). Then, for every $\delta\in\c\setminus\{0\}$
the function $x_\delta:=x+\delta h\in \Oscr(R)$ satisfies 
\begin{equation}\label{eq:dx1ne0}
	dx_\delta(u_j)\ne 0,\quad j=1,2,\ldots.
\end{equation}
Hence, $(x_\delta,y)\colon R\to\C^2$ is an immersion.
We need to correct $x_\delta$ in order to achieve the period vanishing condition in \eqref{eq:x'-conditions}.
To this end, choose the curves $C_1,\ldots,C_\ell$ above 
such that $C:=\bigcup_{k=1}^\ell C_k$ does not contain any of the critical points 
$u_1,u_2,\ldots$ of $y$. Let $\wt x(\cdotp,\zeta)\colon R\to \C$ $(\zeta\in\C^\ell)$ be the spray 
\eqref{eq:X-spray} with the core $\wt x(\cdotp,0)=x$, where the functions $g_k\in \Oscr(R)$ 
are chosen such that the derivative of the period map $\Pcal(\wt x(\cdotp,\zeta),y)$ at $\zeta=0$ 
(see \eqref{eq:derivative-period})  is an isomorphism and $dg_k(u_j)=0$ for all $k=1,2,\ldots,\ell$ and 
$j=1,2,\ldots$. Let $\wt x_\delta(\cdotp,\zeta)\colon R \to \C$ be the holomorphic spray
\[
	\wt x_\delta(u,\zeta) = x_\delta(u) + \sum_{k=1}^\ell \zeta_k\, g_k(u),\quad u\in R,\ \zeta\in\C^\ell.
\]
If $x_\delta$ is sufficiently close to $x$ on the compact set $C$ (which can be achieved by choosing $|\delta|>0$ small enough), 
then the period map $\Pcal(\wt x_\delta(\cdotp,\zeta),y)$ is close to $\Pcal(\wt x(\cdotp,\zeta),y)$.
If the approximation is close enough, the implicit function theorem furnishes a point $\zeta_\delta \in\C^\ell$ 
near $0$ such that $\Pcal(\wt x_\delta(\cdotp,\zeta_\delta),y)=0$ and $\lim_{\delta\to 0} \zeta_\delta =0$.
Fix such $\delta$ and set $x'=\wt x_\delta(\cdotp,\zeta_\delta)\in \Oscr(R)$;  hence the $1$-form $x'dy$ is exact on $R$.
Since $dg_k(u_j)=0$, we also have that $dx'(u_j)= dx_\delta(u_j)\ne 0$ for all $j=1,2,\ldots$ (cf.\ \eqref{eq:dx1ne0}),
so $(x',y)\colon R\to \C^2$ is an immersion. Finally, defining the function $z'\in\Oscr(R)$ by 
\eqref{eq:zprime} we obtain a Legendrian immersion $f'=(x',y,z')\colon R\to \C^3$ 
which approximates $f$ in the $\Cscr^1(M)$-norm. This concludes the first step of the proof.

%
%

It remains to show that a holomorphic Legendrian immersion $f\colon M\to\C^3$ can be approximated in the $\Cscr^1(M)$-norm 
by holomorphic Legendrian embeddings $\wt f\colon M\to\C^3$. 
We shall follow the idea in  \cite[proof of Theorem 4.1]{AlarconForstnericLopez2016MZ}
where the analogous result was shown for holomorphic null curves in $\C^n$ for $n\ge 3$; a similar idea
was also used in \cite[proof of Theorem 2.4]{AlarconForstneric2014IM} to find conformal minimal 
embeddings of open Riemann surfaces into $\R^n$ for any $n\ge 5$.

To a map $f\colon M\to \C^3$ we associate the {\em difference map} $\delta f\colon M\times M\to \C^3$ defined by
\[
	\delta f(u,v)=f(v)-f(u), \quad u,v\in M.
\]
Clearly, $f$ is injective if and only if 
\[
	(\delta f)^{-1}(0)= D_M:=\{(u,u): u\in M\}.
\] 
Assuming that $f$ is an immersion, there is an open neighborhood 
$U\subset M\times M$ of the diagonal  $D_M$ such that $\delta f$ does not assume the value 
$0\in \C^3$ on $\overline U\setminus D_M$. 

Assume that $f\colon M\to\c^3$ is a holomorphic Legendrian immersion.
We shall approximate $f$ in the $\Cscr^1(M)$-norm by a holomorphic Legendrian immersion 
$\wt f \colon M\to\C^3$ whose difference map $\delta \wt f$  is transverse to the origin 
$0\in \C^3$ on $M\times M\setminus U$. Since $\dim M\times M =2 < 3=\dim\C^3$, this will imply that 
$\delta\wt f$ does not assume the value zero on $M\times M\setminus U$, so 
$\wt f(u)\ne \wt f(v)$ if $(u,v)\in M\times M\setminus U$. 
If on the other hand $(u,v)\in U \setminus D_M$, then $\wt f(u)\ne \wt f(v)$ provided that 
$\wt f$ is close enough to $f$; hence the map $\wt f$ is an embedding. 

To find such $\wt f$, it suffices to construct a holomorphic map $H\colon M \times \C^N \to \C^3$ 
for some big integer $N\in\N$ satisfying the following properties for some $r>0$.
\begin{itemize}
\item[\rm (a)] $H(\cdotp,0)=f$.
\vspace{1mm}
\item[\rm (b)] The map $H(\cdotp,\xi)\colon M\to \C^3$ is a holomorphic Legendrian immersion 
for every $\xi \in r\B$, where $\B$ is the unit ball in $\C^N$.
\vspace{1mm}
\item[\rm (c)]  The difference map $\delta H\colon M\times M \times  r\B \to \C^3$, defined by 
\[
	\delta H(u,v,\xi) = H(v,\xi)-H(u,\xi), \quad u,\, v \in M, \  \xi \in r\B,
\]
is a {\em submersive family  of maps} on $M\times M\setminus U$, in the sense that  
\begin{equation} \label{eq:pd}
	\di_\xi|_{\xi=0} \, \delta H(u,v,\xi)   : \C^N \to \C^3 \ \ 
	\text{is surjective for every $(u,v)\in M\times M\setminus U$}. 
\end{equation}
\end{itemize}

Assume for a moment that such $H$ exists. By compactness of $M\times M \setminus U$
it follows from \eqref{eq:pd}  that the partial differential $\di_\xi (\delta H)$ is surjective 
on $(M\times M\setminus U) \times r'\B$ for some $0<r' \le r$. Hence the map 
$\delta H \colon (M\times M\setminus U)\times r'\B \to\C^3$ 
is transverse to any submanifold of $\C^3$, in particular, to the origin $0\in \C^3$. 
By Abraham's reduction to Sard's theorem \cite{Abraham1963TAMS} 
(see also \cite[Section  7.8]{Forstneric2011book} and the references therein for the 
holomorphic case) if follows that for a generic choice of $\xi\in r'\B$  the difference map 
$\delta H(\cdotp,\cdotp,\xi)$ is transverse to $0\in\C^3$ on $M\times M\setminus U$,
and hence it omits the value $0$ by dimension reasons. 
Choosing $\xi$ sufficiently close to $0\in\C^N$ we obtain a holomorphic 
Legendrian embedding $\wt f = H(\cdotp,\xi)\colon M \to \C^3$ close to $f$ in the $\Cscr^1(M)$-norm.

The main point is to find for any given point $(p,q)\in M\times M\setminus U$ 
a spray $H$ as above, with $N=3$, such that \eqref{eq:pd} holds
at $(u,v)=(p,q)$. Since the submersivity of the differential is an open condition and 
$M\times M\setminus U$ is compact, we obtain a spray $H$ that is submersive
at all points of $M\times M\setminus U$ by composing finitely many such sprays 
as explained in \cite[proof of Theorem 4.1]{AlarconForstnericLopez2016MZ}
or \cite[proof of Theorem 2.4]{AlarconForstneric2014IM}.

Fix a pair of distinct points $p\ne q$  in $M$. Choose a smooth embedded
arc $E \subset M$ connecting $p$ to $q$. As above, there exist smooth closed curves $C_1,\ldots,C_\ell\subset M$ 
forming a basis of the homology group $H_1(M;\Z)=\Z^\ell$ such that 
$\big(\bigcup_{k=1}^\ell C_k\big) \cap E = \emptyset$ and $\big(\bigcup_{k=1}^\ell C_k\big) \cup E$
is Runge in $M$. Given a number $\mu >0$, we choose holomorphic functions 
$g_1,\ldots, g_\ell, h_1,h_2 \in \Oscr(M)$ satisfying the following conditions:
\begin{itemize}
\item[\rm (i)] $\int_{C_j} g_k \, dy\approx \delta_{j,k}\ \ \text{for all}\ \ j,k=1,\ldots,\ell$;
\vspace{1mm}
\item[\rm (ii)] $|g_k(u)| < 1$ for all $u\in E$ and $k=1,\ldots,\ell$;
\vspace{1mm}
\item[\rm (iii)] $h_1(p)=0$, $h_1(q)=1$, $h_2(p)=h_2(q)=0$;
\vspace{1mm}
\item[\rm (iv)]  $\int_E h_2\, dy=-1$;
\vspace{1mm}
\item[\rm (v)]  $|h_j(u)| < \mu$ and $|dh_j(u)| < \mu$ for all $u\in \bigcup_{k=1}^\ell C_k$ and $j=1,2$.
\end{itemize}
Functions with these properties are easily found by first constructing suitable smooth 
functions on the curves $\big(\bigcup_{k=1}^\ell C_k\big) \cup E$ and applying 
Mergelyan's approximation theorem.

Let $\xi=(\xi_1,\xi_2,\xi_3)\in \C^3$ and $\zeta=(\zeta_1,\ldots,\zeta_\ell)\in \C^\ell$.
Consider the following sprays of holomorphic functions of $u\in M$:
\begin{eqnarray} 
	\wt x(u,\xi,\zeta) &=& x(u)+ \xi_1 h_1(u) + \xi_3 h_2(u) +  \sum_{k=1}^\ell \zeta_k\, g_k(u), 
	\label{eq:spraysxy1} \\
	\wt y(u,\xi)         &=& y(u) + \xi_2 h_1(u).  \label{eq:spraysxy2}
\end{eqnarray}
Note that $\wt x(u,0,0)=x(u)$ and $\wt y(u,0)=y(u)$. Condition (i) ensures that 
\[
	\frac{\di}{\di\zeta}\bigg|_{\zeta=0} \Pcal\bigl(\wt x(\cdotp,0,\zeta),\wt y(\cdotp,0)\bigr)  : \C^\ell \lra \C^\ell
	\quad \text{is an isomorphism}.
\]
Therefore, by the implicit function theorem, the period vanishing equation 
\begin{equation}\label{eq:periodvanishing}
	\Pcal\bigl(\wt x(\cdotp,\xi,\zeta),\wt y(\cdotp,\xi)\bigr) = 
	\left( \int_{C_k} \wt x(\cdotp,\xi,\zeta)\, d\wt y(\cdotp,\xi) \right)_{k=1,\ldots,\ell}  =0 	
\end{equation}
(which holds at $\xi=0$ and $\zeta=0$) can be solved in the form $\zeta=\rho(\xi)$,
where $\rho$ is a holomorphic map from a neighborhood of $0\in \C^3$ to a neighborhood
of $0\in \C^\ell$ with $\rho(0)=0$. We must estimate the differential $d\rho(\xi)$ at $\xi=0$. 
Differentiating \eqref{eq:periodvanishing} gives in view of $\zeta=\rho(\xi)$ and the chain rule that
\[
	\frac{\di}{\di\zeta}\big|_{\zeta=0} \Pcal\bigl(\wt x(\cdotp,0,\zeta),\wt y(\cdotp,0)\bigr) \cdotp
	d\rho(0) = - \frac{\di}{\di\xi}\big|_{\xi=0} \Pcal\bigl(\wt x(\cdotp,\xi,0),\wt y(\cdotp,\xi)\bigr).
\]
From \eqref{eq:spraysxy1}, \eqref{eq:spraysxy2} and condition (v) on $h_1$ and $h_2$ we see that 
the right hand side of the above equation is of size $O(\mu)$.
Indeed, its components are integrals over the curves $C_1,\ldots, C_\ell$ of terms which involve at least one of the functions 
$h_1,h_2$ or $dh_1$; these are of size $O(\mu)$ by condition (v).
Since the first term on the left hand side is close to the identity, we get
\begin{equation}\label{eq:estimatedrho}
	|d\rho(0)| = O(\mu).
\end{equation}

Define the holomorphic spray $\wt z(\cdotp,\xi)$ with the core $\wt z(\cdotp,0)=z$ on $M$ by
\begin{equation}\label{eq:tildez}
	\wt z(u,\xi)= z(p) - \int_{t=p}^{t=u} \wt x(t,\xi,\rho(\xi))\, d\wt y(t,\xi),\quad u\in M.
\end{equation}
Note that the integral is independent of the choice of the path from $p$ to $u$. It follows that
the spray $H(\cdotp,\xi)\colon M\to \C^3$ defined by
\[
	H(u,\xi) = \bigl( \wt x(u,\xi,\rho(\xi)), \wt y(u,\xi), \wt z(u,\xi) \bigr), \quad u\in M, 
\] 
consists of Legendrian maps and is holomorphic with respect to $\xi\in\C^3$ near the origin. 
Obviously, $H$ satisfies properties (a) and (b). It remains to see that it also satisfies property (c), i.e., 
\begin{equation}\label{eq:deltaHiso}
	\di_\xi|_{\xi=0} \, \delta H(p,q,\xi)  : \C^3 \to \C^3 \ \ \text{is an isomorphism}.
\end{equation}
Condition (iii) on $h_1$ and $h_2$ implies that
\begin{equation}\label{eq:xyxi}
	\frac{\di}{\di \xi}\big|_{\xi=0} \delta \wt x(p,q,\xi,0) = (1,0,0),\quad 
	\frac{\di}{\di \xi}\big|_{\xi=0} \delta \wt y(p,q,\xi,0) = (0,1,0).
\end{equation}
We claim that the following estimates hold:
\begin{eqnarray}
	\frac{\di}{\di \xi}\big|_{\xi=0}\, \delta \wt x(p,q,\xi,\rho(\xi)) &=& (1,0,0) + O(\mu), \label{eq:xxi} \\
	\frac{\di}{\di \xi}\big|_{\xi=0}\, \delta \wt y(p,q,\xi,\rho(\xi)) &=& (0,1,0)  + O(\mu), \label{eq:yxi} \\
	\frac{\di}{\di \xi_3}\big|_{\xi=0}\, \delta \wt z(p,q,\xi)          &=& 1  + O(\mu).   \label{eq:zxi}
\end{eqnarray}
The first two follow directly from the definitions of $\wt x$ and $\wt y$ (see
\eqref{eq:spraysxy1} and \eqref{eq:spraysxy2}) and taking into account
\eqref{eq:estimatedrho} and \eqref{eq:xyxi}.  To get \eqref{eq:zxi}, note that \eqref{eq:tildez} implies
\[
	\delta \wt z(p,q,\xi)=  - \int_{t=p}^{t=q} \wt x(t,\xi,\rho(\xi))\, d\wt y(t,\xi)
\]
where the integral is taken over the arc $E$. From conditions (ii), (iv), \eqref{eq:spraysxy1}, \eqref{eq:spraysxy2} and 
\eqref{eq:estimatedrho} we infer that 
\[
	\frac{\di}{\di \xi_3}\big|_{\xi=0} \delta \wt z(p,q,\xi)  = 
	-\int_E h_2\, dy - \int_E \sum_{k=1}^\ell \frac{\di \rho_k}{\di \xi_3} (0) \, g_k \, dy =1 + O(\mu)
\]
which establishes \eqref{eq:zxi}. Choosing the constant $\mu>0$  small enough, we see from \eqref{eq:xxi}, \eqref{eq:yxi} 
and \eqref{eq:zxi} that \eqref{eq:deltaHiso} holds. By what has been said before, this shows
that $f$ can be approximated on $M$ by Legendrian embeddings.
\end{proof}


\section{Approximation by proper holomorphic Legendrian embeddings}\label{sec:proper}

Given $n\in\n$ and $\sigma\in\{1,\ldots,n\}$ we denote by $\pi_\sigma\colon\c^n\to\c$ the coordinate projection 
\[
	\pi_\sigma(z_1,\ldots,z_n)=z_\sigma. 
\]
If $n>1$ and $\varsigma\in\{1,\ldots,n\}\setminus\{\sigma\}$, we write 
\[
	\pi_{\sigma,\varsigma}=(\pi_\sigma,\pi_\varsigma)\colon\c^n\to\c^2.
\]
In this section we prove the following main theorem of the paper.

\begin{theorem}\label{th:proper}
Let $S=K\cup \Gamma$ be an admissible subset in an open connected Riemann surface $R$ (see Definition \ref{def:admissible}), and let 
$f\colon S\to \C^{2n+1}$ be a generalized Legendrian curve (see Definition \ref{def:generalized}). Also choose two distinct numbers 
$\sigma,\varsigma\in\{1,\ldots,2n+1\}$. Then $f$ may be approximated in the $\Cscr^1(S)$-topology by holomorphic Legendrian 
embeddings $\wt f\colon R\hra \c^{2n+1}$ such that the projection $\pi_{\sigma,\varsigma}\circ f\colon R\to\c^2$ is a proper map.
\end{theorem}

The next result will be the key to obtain the properness condition in Theorem \ref{th:proper}.

\begin{lemma}\label{lem:proper}
Let $M_1$ and $M_2$ be smoothly bounded compact domains in an open Riemann surface $R$ such that $M_1\subset \mathring M_2$ 
and $M_1$ is a strong deformation retract of $M_2$. Choose two distinct numbers $\sigma,\varsigma\in\{1,\ldots,2n+1\}$ $(n\in\n)$, let 
$f\colon M_1\to\c^{2n+1}$ be a Legendrian curve of class $\Ascr^1(M_1)$, and assume that 
\begin{equation}\label{eq:max}
	\max\{|\pi_\sigma\circ f|,|\pi_\varsigma\circ f|\}>\mu\quad \text{on $bM_1$}
\end{equation}
for some $\mu>0$. Then $f$ may be approximated in the $\Cscr^1(M_1)$-topology by holomorphic Legendrian embeddings 
$\wt f\colon M_2\to\c^{2n+1}$ enjoying the following properties:
\begin{enumerate}[\rm (i)]
\item $\max\{|\pi_\sigma\circ \wt f|,|\pi_\varsigma\circ \wt f|\}>\mu+1$ on $bM_2$;
\vspace{1mm}
\item $\max\{|\pi_\sigma\circ \wt f|,|\pi_\varsigma\circ \wt f|\}>\mu$ on $M_2\setminus\mathring M_1$.
\end{enumerate} 
\end{lemma}
\begin{proof}
Let $\alpha_1,\ldots, \alpha_k$ $(k\in\n)$ be pairwise disjoint smooth Jordan curves which are the 
connected components of $b M_1$. Likewise, let $\beta_1,\ldots, \beta_k$ be the pairwise disjoint smooth Jordan curves in 
$b M_2$, labeled so that $M_2\setminus \mathring M_1$ consists of $k$ closed annuli $A_1,\ldots, A_k$ with boundaries 
$bA_l=\alpha_l\cup \beta_l$, $l=1,\ldots,k$; take into account that $M_1$ is a strong deformation retract of $M_2$.
 Since $\max\{|\pi_\sigma\circ f|,|\pi_\varsigma\circ f|\}>\mu$ on $bM_1$ (cf.\ \eqref{eq:max}), 
 there exist an integer $m\ge 2$ and compact connected subarcs $\{\alpha_{l,a} \colon a\in \z_m=\z/m\z\}$ of 
 $\alpha_l$ for $l=1,\ldots,k$ such that the following conditions hold.
\begin{enumerate}[\rm ({A}1)]
\item $\bigcup_{a\in\z_m} \alpha_{l,a}=\alpha_l$.
\vspace{1mm}
\item The arcs $\alpha_{l,a-1}$ and $\alpha_{l,a}$ meet at a point $p_{l,a}$ and are otherwise disjoint for all $a\in \z_m$, 
whereas $\alpha_{l,a}\cap\alpha_{l,a'}=\emptyset$ provided that $a'\notin\{a-1,a,a+1\}\subset \z_m$. 
\vspace{1mm}
\item The set $I:=\{1,\ldots,k\}\times\z_m$ splits into two disjoint subsets $I_\sigma$ and $I_\varsigma$ 
such that 
\[
    \text{$|\pi_c\circ f|>\mu$ on $\alpha_{l,a}$ for all $(l,a)\in I_c$, $c\in\{\sigma,\varsigma\}$.}
\] 
(Either of the sets $I_\sigma$ or $I_\varsigma$ could be empty.) 
\end{enumerate}

Choose a family of pairwise disjoint smooth Jordan arcs $\gamma_{l,a}\subset M_2\setminus\mathring M_1$, $(l,a)\in I$, 
such that $p_{l,a}\in bM_1$ is an endpoint of $\gamma_{l,a}$, the other endpoint of $\gamma_{l,a}$, which will be called $q_{l,a}$, 
lies in $bM_2$, and $\gamma_{l,a}\setminus\{p_{l,a},q_{l,a}\}\subset \mathring M_2\setminus M_1$. Moreover, we choose the 
arcs $\gamma_{l,a}\subset M_2\setminus\mathring M_1$  for $(l,a)\in I$ such that the set
\begin{equation}\label{eq:properS}
      S:=M_1\cup \big (\bigcup_{(l,a)\in I} \gamma_{l,a}\big)
\end{equation}
is admissible in $R$ (see Definition \ref{def:admissible}). Note that the set 
$A_l\setminus (\alpha_l\cup\beta_l\cup (\bigcup_{a\in \z_m} \gamma_{l,a} ))$ consists of $m$ pairwise disjoint open disks; 
we denote by $\Omega_{l,a}$ the one whose closure contains $\alpha_{l,a}$, $(l,a)\in I$. We also set 
$\beta_{l,a}:=\beta_l\cap\overline\Omega_{l,a}$ for all $(l,a)\in I$, and hence $\Omega_{l,a}$ is bounded by the arcs $\alpha_{l,a}$, 
$\gamma_{l,a}$, $\gamma_{l,a+1}$, and $\beta_{l,a}$. (See Figure \ref{fig:Lemma52}.)
\begin{figure}[ht]
    \begin{center}
 \scalebox{0.17}{\includegraphics{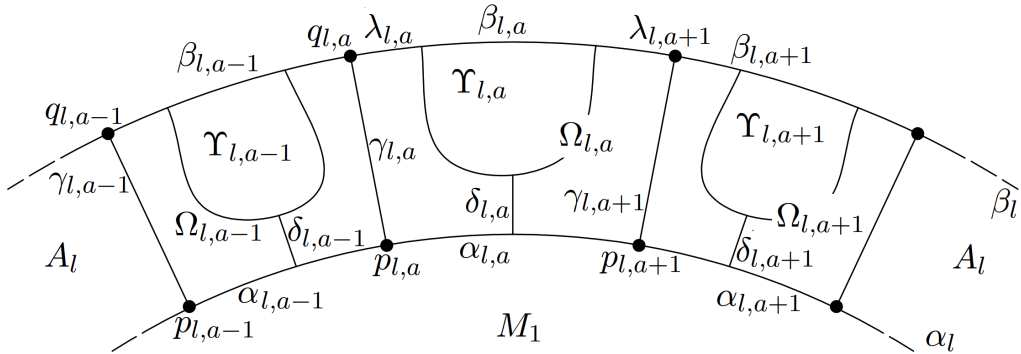}}
     \end{center}
\caption{Sets in the proof of Lemma \ref{lem:proper}.}
\label{fig:Lemma52}
\end{figure}

In the first step we extend the map $f\colon M_1\to\c^{2n+1}$ to a generalized Legendrian curve $\wh g\colon S\to\c^{2n+1}$ 
(see Definition \ref{def:generalized}) such that:
\begin{itemize}
\item $|\pi_c\circ \wh g|>\mu$ on $\gamma_{l,a}\cup\alpha_{l,a}\cup \gamma_{l,a+1}$ for all $(l,a)\in I_c$, $c\in\{\sigma,\varsigma\}$;
\vspace{1mm}
\item $\min\{|\pi_c(\wh g(q_{l,a}))|,|\pi_c(\wh g(q_{l,a+1}))|\}>\mu+1$ for all $(l,a)\in I_c$, $c\in\{\sigma,\varsigma\}$. 
\end{itemize}
To construct such extension we just use property {\rm (A3)} and the fact that every compact path in $\c^{2n+1}$ may be 
uniformly approximated by Legendrian paths (see Theorem \ref{th:path-approx}). 
Lemma \ref{lem:fixedcomponents}  then provides a holomorphic Legendrian curve $g\colon M_2\to\c^{2n+1}$ with no constant 
component function and satisfying the following conditions:
\begin{enumerate}[\rm ({B}1)]
\item the map $g$ approximates $f$ in the $\Cscr^1(M_1)$-topology;
\vspace{1mm}
\item $|\pi_c\circ g|>\mu$ on $\gamma_{l,a}\cup\alpha_{l,a}\cup\gamma_{l,a+1}$ for all $(l,a)\in I_c$, $c\in\{\sigma,\varsigma\}$;
\vspace{1mm}
\item $\min\{|\pi_c(g(q_{l,a}))|,|\pi_c(g(q_{l,a+1}))|\}>\mu+1$ for all $(l,a)\in I_c$, $c\in\{\sigma,\varsigma\}$.
\end{enumerate}
Let $\Omega\subset M_2$ be a compact domain meeting the following requirements:
\begin{itemize}
\item  $S\subset \Omega$ and $S\cap b\Omega=\{q_{l,a}\colon (l,a)\in I\}$.
\vspace{1mm}
\item The set $b\Omega \cap \beta_l$ consists of $m$ pairwise disjoint compact arcs $\{\lambda_{l,a} \colon  a\in \z_m\}$ such that 
$\lambda_{l,a}\subset\beta_{l,a-1}\cup\beta_{l,a}$ and $q_{l,a}$ is in the relative interior of $\lambda_{l,a}$ for all $a\in \z_m$, $l=1,\ldots,k$.
\vspace{1mm}
\item The set $\overline{A_l\setminus \Omega}$ consists of  $m$ pairwise disjoint, smoothly bounded, simply connected compact domains 
$\Upsilon_{l,a}$, $a\in\z_m$, $l=1,\ldots,k$. Up to relabeling, we assume that $\Upsilon_{l,a}\subset\overline\Omega_{l,a}$ for all $(l,a)\in I$. 
Note that $\Upsilon_{l,a}\cap (\gamma_{l,a}\cup\alpha_{l,a}\cup\gamma_{l,a+1})=\emptyset$.
\end{itemize}
Furthermore, in view of {\rm (B2)} and {\rm (B3)}, continuity of $g$ enables us to choose $\Omega$ sufficiently close to $S$ 
such that the following properties hold:
\begin{enumerate}[\rm ({C}1)]
\item $|\pi_c\circ g|>\mu$ on $\overline{\Omega_{l,a}\setminus \Upsilon_{l,a}}$ for all $(l,a)\in I_c$, $c\in\{\sigma,\varsigma\}$.
\vspace{1mm}
\item $|\pi_c\circ g|>\mu+1$ on $(\lambda_{l,a}\cup\lambda_{l,a+1})\cap\beta_{l,a}$ for all $(l,a)\in I_c$, $c\in\{\sigma,\varsigma\}$.
\end{enumerate}

Assume that $I_\sigma\neq\emptyset$; otherwise $I_\varsigma=I\neq\emptyset$ and we reason in a symmetric way. 
For each $(l,a)\in I_\sigma$ pick an arc $\delta_{l,a}\subset \overline{\Omega_{l,a}\setminus \Upsilon_{l,a}}$ with an endpoint in the relative interior of $\alpha_{l,a}$, the other endpoint in $\Upsilon_{l,a}\setminus \beta_{l,a}$, and otherwise disjoint from 
$\Upsilon_{l,a}\cup b\overline\Omega_{l,a}$. (See Figure \ref{fig:Lemma52}.) 
Moreover, we choose the family of arcs $\delta_{l,a}$, $(l,a)\in I_\sigma$, such that the set
\begin{equation}\label{eq:properSsigma}
      S_\sigma:=S\cup \big( \bigcup_{(l,a)\in I_\varsigma} \overline \Omega_{l,a}\big) \cup \big (\bigcup_{(l,a)\in I_\sigma} \delta_{l,a}\cup \Upsilon_{l,a}\big)\quad \text{is admissible in $R$}
\end{equation}
(see \eqref{eq:properS}) and 
\begin{equation}\label{eq:nozerosg}
     \text{no component function of $g$ has a zero 
     or a critical point in $\bigcup_{(l,a)\in I_\sigma} \delta_{l,a}$}.
\end{equation}
The latter condition is easy to fulfil since $g$ is holomorphic and has no constant component function. 
Notice that $S_\sigma\subset M_2$ is a deformation retract of $M_2$. Set 
\[
	K_\sigma:=\overline{\mathring S_\sigma},\qquad
	\Gamma_\sigma:=\overline{S_\sigma\setminus K_\sigma}=\bigcup_{(l,a)\in I_\sigma} \delta_{l,a}
\]
(cf.\ Definition \ref{def:admissible}). Let us prove the following

\begin{claim}\label{cl:hi}
There exists a generalized Legendrian curve $\wt h_\sigma\colon S_\sigma\to\c^{2n+1}$ satisfying the following properties:
\begin{enumerate}[\rm ({D}1)]
\item $\wt h_\sigma= g$ on $S\cup \big( \bigcup_{(l,a)\in I_\varsigma} \overline \Omega_{l,a}\big)$;
\vspace{1mm}
\item $\pi_\sigma\circ\wt h_\sigma=\pi_\sigma\circ g|_{S_\sigma}$;
\vspace{1mm}
\item $|\pi_\varsigma\circ\wt h_\sigma|>\mu+1$ on $\Upsilon_{l,a}$ for all $(l,a)\in I_\sigma$;
\vspace{1mm}
\item if $\sigma=2n+1$ then $\pi_1\circ\wt h_\sigma$ and $\pi_2\circ\wt h_\sigma$ are constant on no component of $K_\sigma$,  
$\pi_1\circ\wt h_\sigma$ has no zeros in $\Gamma_\sigma$ and $\pi_2\circ\wt h_\sigma$ has no critical points in $\Gamma_\sigma$.
\end{enumerate}
\end{claim}
\begin{proof}
If $2n+1\notin\{\sigma,\varsigma\}$ then, by Remark \ref{rem:etaprime}, we may assume that $\sigma=1$; 
note that the transformation in that remark preserves the norm of the first $2n$ components. Thus, it suffices to discuss the cases
$(\sigma,\varsigma)\in\{1\}\times\{1,\ldots,2n\}$ and
 $(\sigma,\varsigma)\in \big(\{1,2\}\times \{2n+1\}\big)\cup \big(\{2n+1\}\times\{1,2\}\big)$. Write $g=(a_1,b_1,\ldots,a_n,b_n,w)$.

\smallskip
\noindent{\em Case 1: $(\sigma,\varsigma)\in\{1\}\times\{1,\ldots,2n\}$.} 
Let $g'=(a_1',b_1',\ldots,a_n',b_n',w')\colon S_1\to\c^{2n+1}$ be any map of class $\Ascr^1(S_1)$ (see \eqref{eq:ArS}) 
satisfying the following conditions:
\begin{enumerate}[\rm (a)]
\item $g'=g$ on $S\cup \big( \bigcup_{(l,a)\in I_\varsigma} \overline \Omega_{l,a}\big)$;
\vspace{1mm}
\item $a_1'=a_1|_{S_1}$;
\vspace{1mm}
\item $|\pi_\varsigma \circ g'|>\mu+1$ on $\Upsilon_{l,a}$ for all $(l,a)\in I_1$.
\end{enumerate}
The existence of such a map is trivial. In principle, $g'$ does not need to be Legendrian. 
Choose an initial point $u_0\in\mathring M_1$ and define $w''\colon S_1\to\c$ by $w''=w'=w$ on 
$S\cup \big( \bigcup_{(l,a)\in I_\varsigma} \overline \Omega_{l,a}\big)$ and
\[
	w''(u)=w(u_0)-\int_{u_0}^u \sum_{j=1}^n a_j'db_j',\quad u\in S_1.
\]
Conditions {\rm (a)} and {\rm (b)}, together with the facts that $g$ is a Legendrian curve and the set 
\[
	S_1\setminus \big(S\cup \big( \bigcup_{(l,a)\in I_\varsigma} \overline \Omega_{l,a}\big)\big)=
	\bigcup_{(l,a)\in I_1} \delta_{l,a}\cup \Upsilon_{l,a}
\] 
is simply connected, ensure that $w''$ is well defined and of class $\Ascr^1(S_1)$.
Thus, conditions {\rm (a)}, {\rm (b)}, and {\rm (c)} guarantee that the map $\wt h_1:=(a_1',b_1',\ldots,a_n',b_n',w'')\colon S_1\to\c^{2n+1}$ 
is a generalized Legendrian curve satisfying properties {\rm (D1)}, {\rm (D2)}, and {\rm (D3)}. Property {\rm (D4)} is vacuous in this case.

\smallskip
\noindent{\em Case 2: $(\sigma,\varsigma)=(1,2n+1)$.} Choose any map
$g'=(a_1',b_1',\ldots,a_n',b_n',w')\colon S_1\to\c^{2n+1}$ of class $\Ascr^1(S_1)$ with the following properties:
\begin{enumerate}[\rm (a)]
\item $g'=g$ on $S\cup \big( \bigcup_{(l,a)\in I_{2n+1}} \overline \Omega_{l,a}\big)$;
\vspace{1mm}
\item $a_1'=a_1|_{S_1}$;
\vspace{1mm}
\item $|w'|>\mu+1$ on $\Upsilon_{l,a}$ for all $(l,a)\in I_1$;
\vspace{1mm}
\item $dw'+\sum_{j=2}^n a_j'db_j'$ vanishes nowhere on $\bigcup_{(l,a)\in I_1}\delta_{l,a}$ and its zeros on 
$\bigcup_{(l,a)\in I_1} \Upsilon_{l,a}$ are those of $a_1'=a_1$, with the same order.
\end{enumerate}
The existence of such $g'$ is clear; for property {\rm (d)} take into account \eqref{eq:nozerosg}. 
Now choose an initial point $u_0\in\mathring M_1$ and define $b_1''\colon S_1\to\c$ by 
$b_1''=b_1'=b_1$ on $S\cup \big( \bigcup_{(l,a)\in I_{2n+1}} \overline \Omega_{l,a}\big)$ and
\[
	b_1''(u)=b_1(u_0)-\int_{u_0}^u \frac{dw'+\sum_{j=2}^n a_j'db_j'}{a_1'},\quad u\in S_1;
\]
recall that, by \eqref{eq:nozerosg} and property {\rm (b)}, $a_1'$ has no zeros in $\bigcup_{(l,a)\in I_1}\delta_{l,a}$. It follows that the map 
$\wt h_1:=(a_1',b_1'',\ldots,a_n',b_n',w')\colon S_1\to\c^{2n+1}$ is a generalized Legendrian curve satisfying 
{\rm (D1)}, {\rm (D2)}, and {\rm (D3)}; property {\rm (D4)} is again vacuous.

\smallskip
\noindent{\em Case 3: $(\sigma,\varsigma)=(2,2n+1)$.} 
Choose any map $g'=(a_1',b_1',\ldots,a_n',b_n',w')\colon S_2\to\c^{2n+1}$ of class $\Ascr^1(S_2)$ 
enjoying the following properties:
\begin{enumerate}[\rm (a)]
\item $g'=g$ on $S\cup \big( \bigcup_{(l,a)\in I_{2n+1}} \overline \Omega_{l,a}\big)$;
\vspace{1mm}
\item $b_1'=b_1|_{S_2}$;
\vspace{1mm}
\item $|w'|>\mu+1$ on $\Upsilon_{l,a}$ for all $(l,a)\in I_2$;
\vspace{1mm}
\item $dw'+\sum_{j=2}^n a_j'db_j'$ vanishes nowhere on $\bigcup_{(l,a)\in I_2}\delta_{l,a}$ and its zeros on 
$\bigcup_{(l,a)\in I_2}\Upsilon_{l,a}$ are those of $db_1'=db_1$, with the same order.
\end{enumerate}
Now choose an initial point $u_0\in\mathring M_1$ and define $a_1''\colon S_2\to\c$ by $a_1''=a_1'=a_1$ on 
$S\cup \big( \bigcup_{(l,a)\in I_{2n+1}} \overline \Omega_{l,a}\big)$ and
\[
	a_1''(u)=-\frac{dw'+\sum_{j=2}^n a_j'db_j'}{db_1'},\quad u\in S_1.
\]
The map $\wt h_2:=(a_1'',b_1',\ldots,a_n',b_n',w')\colon S_2\to\c^{2n+1}$ is a generalized Legendrian curve 
satisfying Claim \ref{cl:hi}.

\smallskip
\noindent{\em Case 4: $(\sigma,\varsigma)=(2n+1,1)$.} Pick a map 
$g'=(a_1',b_1',\ldots,a_n',b_n',w')\colon S_{2n+1}\to\c^{2n+1}$ of class $\Ascr^1(S_{2n+1})$ meeting the following requirements.
\begin{enumerate}[\rm (a)]
\item $g'=g$ on $S\cup \big( \bigcup_{(l,a)\in I_1} \overline \Omega_{l,a}\big)$;
\vspace{1mm}
\item $w'=w|_{S_{2n+1}}$;
\vspace{1mm}
\item $|a_1'|>\mu+1$ on $\Upsilon_{l,a}$ for all $(l,a)\in I_{2n+1}$ and $a_1'$ vanishes nowhere on 
$\bigcup_{(l,a)\in I_{2n+1}}(\delta_{l,a}\cup \Upsilon_{l,a})$;
\vspace{1mm}
\item $dw'+\sum_{j=2}^n a_j'db_j'$ has no zeros on $\bigcup_{(l,a)\in I_{2n+1}}\delta_{l,a}$.
\end{enumerate}
Also, pick an initial point $u_0\in\mathring M_1$ and define $b_1''\colon S_{2n+1}\to\c$ as in Case 2. 
Thus, the map $\wt h_{2n+1}:=(a_1',b_1'',\ldots,a_n',b_n',w')\colon S_{2n+1}\to\c^{2n+1}$ satisfies the claim. 
Notice that {\rm (c)} and {\rm (d)} imply property {\rm (D4)}.

\smallskip
\noindent{\em Case 5: $(\sigma,\varsigma)=(2n+1,2)$.} Take a map $g'\colon S_{2n+1}\to\c^{2n+1}$ of class 
$\Ascr^1(S_{2n+1})$ meeting the following requirements: 
\begin{enumerate}[\rm (a)]
\item $g'=g$ on $S\cup \big( \bigcup_{(l,a)\in I_2} \overline \Omega_{l,a}\big)$;
\vspace{1mm}
\item $w'=w|_{S_{2n+1}}$;
\vspace{1mm}
\item $|b_1'|>\mu+1$ on $\Upsilon_{l,a}$ for all $(l,a)\in I_{2n+1}$ and $db_1'$ vanishes nowhere on 
$\bigcup_{(l,a)\in I_{2n+1}}(\delta_{l,a}\cup \Upsilon_{l,a})$;
\item $dw'+\sum_{j=2}^n a_j'db_j'$ has no zeros on $\bigcup_{(l,a)\in I_{2n+1}}\delta_{l,a}$.
\end{enumerate}
Fix a point $u_0\in\mathring M_1$ and define $a_1''\colon S_{2n+1}\to\c$ as in Case 3. Again, the map 
$\wt h_{2n+1}:=(a_1'',b_1',\ldots,a_n',b_n',w')\colon S_{2n+1}\to\c^{2n+1}$ meets all the requirements. 
As in the previous case, requirements {\rm (c)} and {\rm (d)} ensure property {\rm (D4)}.
\end{proof}

With Claim \ref{cl:hi} in hand, and since $g$ has no constant component function, we may apply 
Lemmas \ref{lem:fixedcomponents} and \ref{lem:position} to obtain a holomorphic Legendrian embedding 
$h_\sigma\colon M_2\hra \c^{2n+1}$ with the following properties:
\begin{enumerate}[\rm ({E}1)]
\item $h_\sigma$ approximates $\wt h_\sigma$ in the $\Cscr^1(S_\sigma)$-topology;
\vspace{1mm}
\item $\pi_\sigma\circ h_\sigma=\pi_\sigma\circ g$.
\end{enumerate}
Thus, in view of properties {\rm (D3)} and {\rm (D1)}, if the approximation in property {\rm (E1)} is close enough, we have that:
\begin{enumerate}[\rm ({E}1)]
\item[\rm ({E}3)] $|\pi_\varsigma\circ h_\sigma|>\mu+1$ on $\Upsilon_{l,a}$ for all $(l,a)\in I_\sigma$;
\vspace{1mm}
\item[\rm ({E}4)] $h_\sigma$ has no constant component function.
\end{enumerate}
We claim that if $I_\varsigma=\emptyset$ then $\wt f:=h_\sigma$ satisfies the conclusion of Lemma \ref{lem:proper}. 
Indeed, properties {\rm (B1)} and {\rm (E1)} ensure that $\wt f$ approximates $f$ in the $\Cscr^1(M_1)$-topology. 
If on the other hand $I_\varsigma=\emptyset$, then $I_\sigma=I$ and hence
\[
     bM_2\subset \big( \bigcup_{(l,a)\in I_\sigma} \Upsilon_{l,a}\big)\cup 
     \big( \bigcup_{(l,a)\in I_\sigma} \lambda_{l,a}\big)
\]
and
\[  
    M_2\setminus\mathring M_1=\big( \bigcup_{(l,a)\in I_\sigma} \Upsilon_{l,a}\big)\cup
     \big( \bigcup_{(l,a)\in I_\sigma} \overline{\Omega_{l,a}\setminus \Upsilon_{l,a}}\big).
\]
Thus, properties {\rm (C2)}, {\rm (E2)}, and  {\rm (E3)} give condition {\rm (i)}, whereas properties 
{\rm (C1)}, {\rm (E2)}, and  {\rm (E3)} guarantee condition {\rm (ii)}.

Assume now that $I_\varsigma\neq\emptyset$. As above, for each $(l,a)\in I_\varsigma$, we take an arc 
$\delta_{l,a}\subset \overline{\Omega_{l,a}\setminus \Upsilon_{l,a}}$ with an endpoint in the relative interior of 
$\alpha_{l,a}$, the other endpoint in $\Upsilon_{l,a}\setminus \beta_{l,a}$, and otherwise disjoint from 
$\Upsilon_{l,a}\cup b\overline\Omega_{l,a}$, such that the set
\[
      S_\varsigma:=S\cup \big( \bigcup_{(l,a)\in I_\sigma} 
      \overline \Omega_{l,a}\big) \cup \big (\bigcup_{(l,a)\in I_\varsigma} \delta_{l,a}\cup \Upsilon_{l,a}\big)
\]
is admissible in $R$ and no component function of $h_\sigma$ has a zero or a critical point in $\bigcup_{(l,a)\in I_\sigma} \delta_{l,a}$.
Reasoning as above, in a symmetric way, we may construct a holomorphic Legendrian embedding $h_\varsigma\colon M_2\to \c^{2n+1}$
enjoying the following properties:
\begin{enumerate}[\rm ({F}1)]
\item $h_\varsigma$ approximates $h_\sigma$ in the $\Cscr^1(S_\varsigma)$-topology;
\vspace{1mm}
\item $\pi_\varsigma\circ h_\varsigma=\pi_\varsigma\circ h_\sigma$;
\vspace{1mm}
\item $|\pi_\sigma\circ h_\varsigma|>\mu+1$ on $\Upsilon_{l,a}$ for all $(l,a)\in I_\varsigma$.
\end{enumerate}

We claim that, if the approximations in properties {\rm (B1)}, {\rm (D1)}, {\rm (E1)}, and {\rm (F1)} are close enough, the Legendrian curve 
$\wt f:=h_\varsigma\colon M_2\to\c^{2n+1}$ satisfies the conclusion of the lemma. Indeed, the mentioned properties ensure that $\wt f$ 
approximates $f$ in the $\Cscr^1(M_1)$-topology. It remains to check conditions {\rm (i)} and {\rm (ii)}. 
Let $p\in M_2\setminus\mathring M_1$. If $p\in \Upsilon_{l,a}$ for some $(l,a)\in I_\sigma$, we have
\[
    |\pi_\varsigma(\wt f(p))| \stackrel{{\rm (F1)}}{\approx} |\pi_\varsigma(h_\sigma(p))|  \stackrel{{\rm (E3)}}{>}\mu+1,
\]
and together with property {\rm (F3)} we infer that
\begin{equation}\label{eq:1}
      \max\{|\pi_\sigma\circ \wt f|,|\pi_\varsigma\circ \wt f|\}>\mu+1\quad 
      \text{on $\bigcup_{(l,a)\in I}\Upsilon_{l,a}$}.
\end{equation}
On the other hand, if $p\in \overline{\Omega_{l,a}\setminus\Upsilon_{l,a}}$ for some $(l,a)\in I_\sigma$, we have
\begin{equation}\label{eq:2}
    |\pi_\sigma(\wt f(p))| \stackrel{{\rm (F1)}}{\approx} |\pi_\sigma(h_\sigma(p))|  
    \stackrel{{\rm (E2)}}{=}|\pi_\sigma(g(p))|,
\end{equation}
whereas if $p\in \overline{\Omega_{l,a}\setminus\Upsilon_{l,a}}$ for some $(l,a)\in I_\varsigma$, then
\begin{equation}\label{eq:3}
     |\pi_\varsigma(\wt f(p))| \stackrel{{\rm (F2)}}{=}|\pi_\varsigma(h_\sigma(p))|
     \stackrel{{\rm (E1),(D1)}}{\approx} |\pi_\varsigma(g(p))|.
\end{equation}
So, the estimates  \eqref{eq:2} and \eqref{eq:3} together with the inequalities \eqref{eq:1}, {\rm (C1)}, and {\rm (C2)}, and the facts that
\begin{eqnarray*}
      bM_2 & = & \bigcup_{(l,a)\in I} \beta_{l,a} \; \subset 
      \; \bigcup_{(l,a)\in I} \Upsilon_{l,a}\cup (\beta_{l,a}\setminus\Upsilon_{l,a}),
      \\
     M_2\setminus\mathring M_1 & = &  \bigcup_{(l,a)\in I} \overline\Omega_{l,a} \; =\;
      \bigcup_{(l,a)\in I} \Upsilon_{l,a}\cup \overline{\Omega_{l,a}\setminus\Upsilon_{l,a}},
\end{eqnarray*}
guarantee conditions {\rm (i)} and {\rm (ii)}. This concludes the proof.
\end{proof}

%
%

\begin{proof}[Proof of Theorem \ref{th:proper}]
Up to adding  suitable arcs to the admissible set $S$ and extending $f$ to a generalized Legendrian curve in 
the arising admissible set (see Theorem \ref{th:path-approx}), we may assume that $S$ is connected. 
As in the proof of Lemma \ref{lem:fixedcomponents}, we may also assume that $K\neq\emptyset$. 

Let $M_0\subset R$ be a smoothly bounded compact connected domain such that $S\subset\mathring M_0$ and $S$ is a 
strong deformation retract of $M_0$. By Lemmas \ref{lem:fixedcomponents} and \ref{lem:position},
$f$ can be approximated  in the $\Cscr^1(S)$-topology by a holomorphic Legendrian embedding $f_0\colon M_0\to\c^{2n+1}$.
Moreover, up to a slight deformation of $f_0$, we may assume that $\max\{|\pi_\sigma\circ f_0|,|\pi_\varsigma\circ f_0|\}>\mu$ 
on $bM_0$ for some constant $\mu>0$.

Let $M_0\Subset M_1\Subset M_2 \Subset\ldots$ be a sequence of smoothly bounded compact connected Runge domains in 
$R$ such that:
\begin{itemize}
\item  $\bigcup_{k\geq 0} M_k=R$;
\vspace{1mm}
\item the Euler characteristic $\chi(M_k\setminus \mathring M_{k-1})\in \{0,-1\}$ for all $k\in \n$. 
\end{itemize}
Let us construct a sequence $\{f_k\colon M_k\to \c^{2n+1}\}_{k\in \n}$ of Legendrian embeddings meeting the following requirements
for every $k=1,2,\ldots$: 
\begin{enumerate}[\rm (i{$_k$})]
\item $f_k$ approximates $f_{k-1}$ uniformly on $M_{k-1}$.
\vspace{1mm}
\item $\max\{|\pi_\sigma\circ f_k|,|\pi_\varsigma\circ f_k|\}>\mu+k$ on $bM_k$.
\vspace{1mm}
\item $\max\{|\pi_\sigma\circ f_k|,|\pi_\varsigma\circ f_k|\}>\mu+k-1$ on $M_k\setminus\mathring M_{k-1}$.
\end{enumerate}
We reason in a recursive way. Observe that $f_0\colon M_0\to \c^{2n+1}$ satisfies  {\rm (ii$_0$)}, whereas the conditions
{\rm (i$_0$)} and {\rm (iii$_0$)} are vacuous. 
Assume we have found $f_0,\ldots,f_{k-1}$ satisfying the above properties, and let us construct the next map $f_k$.

If $\chi(M_k\setminus \mathring M_{k-1})=0$, then $M_{k-1}$ is a strong retract deformation of $M_k$.  Lemma \ref{lem:proper} 
applied to $M_{k-1}$, $M_k$, $\sigma,\varsigma\in\{1,\ldots,2n+1\}$, $\mu+k-1$, and $f_{k-1}\colon M_{k-1}\to \c^{2n+1}$  provides 
a holomorphic Legendrian embedding $f_k\colon M_{k}\hra \c^{2n+1}$ enjoying the desired properties {\rm (i$_k$)}, {\rm (ii$_k$)}, 
and {\rm (iii$_k$)}.

If $\chi(M_k\setminus \mathring M_{k-1})=-1$, there is a Jordan arc $\alpha \subset \mathring M_k$  with the endpoints 
in $b M_{k-1}$ and otherwise disjoint from $M_{k-1}$,  such that $S':= M_{k-1}\cup \alpha$ is an admissible subset of $R$ and 
a strong deformation retract of $M_k$. 
By Theorem \ref{th:path-approx} we may extend $f_{k-1}$  to a generalized Legendrian map $ f_{k-1}'\colon S' \to \c^{2n+1}$ satisfying 
\[
	\text{$\max\{|\pi_\sigma\circ f_{k-1}'|,|\pi_\varsigma\circ f_{k-1}'|\}>\mu+k-1$ on $\alpha$.}
\]
By Lemmas \ref{lem:fixedcomponents} and  \ref{lem:position} we can approximate $f_{k-1}'$  in the $\Ccal^1(S')$-topology by a  
holomorphic Legendrian embedding $\wt f_{k-1}\colon  M_{k}\to \c^{n+1}$. If the approximation is close enough,  
there is a  smoothly bounded compact domain $ M_{k-1}'$ such that 
$S'\subset \mathring M_{k-1}'\subset M_{k-1}'\subset \mathring M_k$,  $M_{k-1}'$ is a deformation retract of $M_k$, and
\[
	\text{$\max\{|\pi_\sigma\circ \wt f_{k-1}|,|\pi_\varsigma\circ \wt f_{k-1}|\}>\mu+k-1$ on $M_{k-1}'\setminus \mathring M_{k-1}$.}
\]
This reduces the proof to the previous case. This closes the induction.

If the approximations in condition {\rm (i$_k$)} are close enough, the sequence of Legendrian embeddings
$\{f_k\colon M_k\hra \c^{2n+1}\}_{k\in \n}$ converges uniformly on compacts in $R$ to a holomorphic Legendrian embedding 
$\wt f\colon R\hra \c^{2n+1}$. Condition {\rm (i$_k$)} for $k\in \z_+$ ensures that $\wt f|_S$  approximates $f_0$, and hence $f$, 
in the $\Cscr^1(S)$-topology. Finally, $\max\{|\pi_\sigma\circ \wt f|,|\pi_\varsigma\circ \wt f|\}>\mu+k$ on $R\setminus M_k$ for all $k$ 
provided that the approximations in condition {\rm (i$_j$)} are sufficiently close; take into account conditions {\rm (ii$_j$)} and {\rm (iii$_j$)}, $j\in \n$. 
In particular,  $\pi_{\sigma,\varsigma}\circ \wt f\colon R\to\c^2$ is a proper map.

This concludes the proof of Theorem \ref{th:proper}.
\end{proof}


\section{Complete holomorphic Legendrian curves with Jordan boundaries}\label{sec:Jordan}

In this section we prove Theorem \ref{th:intro-Jordan} in the introduction. For simplicity of exposition we shall prove the theorem in 
the particular case when the source compact bordered Riemann surface $M$ is the closed unit disk $\cd\subset\c$ 
(see Theorem \ref{th:Jordan} below). We point out that the same proof applies in the general case; for that one simply uses 
Theorem \ref{th:RH} (the Riemann-Hilbert problem for holomorphic Legendrian curves normalized by any given bordered Riemann surface)
as opposed to Lemma \ref{lem:RH} (the Riemann-Hilbert problem for holomorphic Legendrian disks). We leave the details to the 
interested reader and refer to \cite[\textsection 4]{AlarconDrinovecForstnericLopez2015PLMS} where complete details are given 
for any bordered Riemann surface $M$ in a geometrically similar situation of conformal minimal immersions and holomorphic null curves.

We begin by pointing out the existence of a large family of Legendrian curves which are contained in certain complex affine hyperplanes 
of $\c^{2n+1}$. These curves will be crucial in the subsequent construction.

%
%
\begin{proposition}\label{pro:flat}
Let $(a_1,b_1,\ldots,a_n,b_n)\in\c^{2n}$ $(n\in\n)$ be such that $a_1\cdots a_n\neq 0$  and denote by $\Pi\subset\c^{2n+1}$ 
the complex vectorial hyperplane 
\[
      \Pi:=\biggl\{(x_1,y_1,\ldots,x_n,y_n,z)\in\c^{2 n+1} : z+\sum_{j=1}^n (a_jx_j+b_jy_j)=0\biggr \}.
\]
Given a point $p_0=(x_{0,1},y_{0,1},\ldots,x_{0,n},y_{0,n},z_0)\in\c^{2 n+1}$, the map
\[
\c\ni\zeta\mapsto \Psi(\zeta)=\big(X_1(\zeta),Y_1(\zeta),\ldots,X_n(\zeta),Y_n(\zeta),Z(\zeta)\big)\in\c^{2 n+1}
\]
defined by
\begin{eqnarray*}
     Y_j(\zeta) & = & y_{0,j}+\zeta,
		\\
     X_j(\zeta) & = & \frac{x_{0,j}-b_j}{e^{y_{0,j}/a_j}} e^{Y_j(\zeta)/a_j}+b_j,
		\\
     Z(\zeta) & = & z_0+\sum_{j=1}^n \big(a_jx_{0,j}+b_jy_{0,j} -a_j X_j(\zeta)-b_j Y_j(\zeta)\big),\quad j=1,\ldots,n,
\end{eqnarray*}
is a proper holomorphic Legendrian embedding $\c\hra\c^{2n+1}$ such that $\Psi(0)=p_0$ and $\Psi(\c)\subset p_0+\Pi$. 
Moreover, $\Psi$ depends holomorphically on $p_0\in\c^{2n+1}$.
\end{proposition}
\begin{proof}
Obviously, $\Psi$ is holomorphic and depends holomorphically on $p_0\in\c^{2n+1}$. A trivial computation shows that 
$dZ_j+\sum_{j=1}^n X_jdY_j=0$ everywhere on $\c$ and so $\Psi$ is Legendrian. Further, $Y_j\colon\c\to\c$ is a proper (linear) 
embedding for each $j\in \{1,\ldots,n\}$ and hence $\Psi$ is a proper embedding as well. Finally, the facts that $\Psi(0)=p_0$ 
and $\Psi(\c)\subset p_0+\Pi$ can be checked by direct computation.
\end{proof}

A holomorphic disk $f\colon \D\to\C^N$ $(N\in\n)$ is said to be {\em complete} if for any path $\gamma\colon [0,1)\to \D$ 
with $\lim_{t\to 1}|\gamma(t)|=1$ the path $f\circ \gamma\colon [0,1)\to \C^N$ has infinite length. Equivalently, denoting
by $ds^2$ the Euclidean metric on $\C^N$, the pull-back $f^*ds^2$ is a complete metric on $\D$. 

If $f\colon\cd\to\c^N$ is an immersion of class $\Ascr^1(\cd)$, we denote by $\dist_f\colon\cd\times\cd\to\r$ the distance on $\cd$ 
associated to the Riemannian metric $f^*ds^2$.

The following is the main result in this section.
%
%
\begin{theorem}\label{th:Jordan}
Every Legendrian curve $f\colon \cd\to\c^{2n+1}$ $(n\in\n)$ of class $\Ascr^1(\cd)$ may be approximated uniformly on $\cd$ 
by continuous injective maps $\wt f\colon\cd\to\c^{2n+1}$ such that $\wt f|_\d\colon\d\to\c^{2n+1}$ is a complete 
holomorphic Legendrian embedding.
\end{theorem}

Recall that $\t=b\cd=\{\zeta\in\c : |\zeta|=1\}$. Most of the technical part in the proof of Theorem \ref{th:Jordan} 
is provided by the following lemma.

%
%
\begin{lemma}\label{lem:Jordan0}
Let $f\colon \cd\to\c^{2n+1}$ $(n\in\n)$ be a Legendrian immersion of class $\Ascr^1(\cd)$, $\Ygot\colon \t\to\c^{2n+1}$ be a smooth map, 
and $\delta>0$ and $d>0$ be numbers such that the following conditions hold:
\begin{enumerate}[\rm (i)]
\item $\|f-\Ygot\|_{0,\t}<\delta$.
\vspace{1mm}
\item $\dist_f(0,\t)>d$.
\end{enumerate}
Given $\mu>0$, $f$ may be approximated uniformly on compacts in $\d$ by Legendrian embeddings $\wt f\colon\cd\hra\c^{2n+1}$ 
of class $\Oscr(\cd)$ satisfying the following properties:
\begin{enumerate}[\rm (I)]
\item $\|\wt f-\Ygot\|_{0,\t}<\sqrt{\delta^2+\mu^2}$.
\vspace{1mm}
\item $\dist_{\wt f}(0,\t)>d+\mu$.
\end{enumerate}
\end{lemma}
\begin{proof}
We assume that $n=1$; the same proof applies in general.

Choose numbers $\epsilon_0>0$ and $0<r_0<1$. To prove the lemma it suffices to find a Legendrian embedding $\wt f\colon\cd\hra\c^3$ 
of class $\Oscr(\cd)$ satisfying properties {\rm (I)}, {\rm (II)}, and $\|\wt f-f\|_{0,r_0\cd}<\epsilon_0$. 

Write $f=(f_1,f_2,f_3)$ and $\Ygot=(\Ygot_1,\Ygot_2,\Ygot_3)$. 
By condition {\rm (ii)} and up to enlarging $r_0<1$ if necessary, there is  a number $d_0$ such that
\begin{equation}\label{eq:Jordanr}
      \dist_f(0,r_0\t)>d_0>d.
\end{equation}
Set $g=(g_1,g_2,g_3):=f-\Ygot\colon\t\to\c^3$ and, up to slightly deforming $\Ygot$, assume that
\begin{equation}\label{eq:f1-Y1}
       \text{$g_1g_3\colon\t\to\c$ has no zeros}.
\end{equation}
By dimension reasons, this can be  achieved without loss of generality. In particular, $g_1$ and $g_3$ 
(and hence $g$) vanish nowhere on $\t$. Further, condition {\rm (i)} gives a number $\delta_0$ such that
\begin{equation}\label{eq:Jordandelta0}
      \|f-\Ygot\|_{0,\t}<\delta_0<\delta.
\end{equation}

Pick a number $\epsilon>0$ which will be specified later. Let $m\in\n$ be a large enough integer such that, setting 
\[
    \alpha_j:=\left\{e^{\imath 2\pi t} : t \in\Big[\frac{j-1}{m}, \frac{j}{m}\Big]\right\}\subset\t,\quad j\in\z_m,
\]
the following estimates are satisfied for all $u,u'\in\alpha_j$ and $j\in\z_m$:
\begin{equation}\label{eq:f-Y}
|f(u)-f(u')|<\epsilon,\quad |f(u)-\Ygot(u')|<\delta_0,\quad |\Ygot(u)-\Ygot(u')|<\epsilon.
\end{equation}
Such $m$ exists in view of \eqref{eq:Jordandelta0} and continuity of $f$ and $\Ygot$. For every $j\in\z_m$ we set 
\begin{equation}\label{eq:endpoints}
      u_j:=e^{\imath 2\pi \frac{j}{m}};
\end{equation}
hence $u_{j-1}$ and $u_j$ are the endpoints of the arc $\alpha_j\subset \T$.

Given $w\in \c^3\setminus\{0\}$ we denote by
\begin{equation}\label{eq:pi_w}
\pi_w\colon\c^3\to \c w=\{\zeta w : \zeta\in\c\}\subset\c^3
\end{equation}
the Hermitian orthogonal projection onto the complex line $\c w$. The proof of the lemma consists of two different deformation 
procedures. The first one is provided by the following.

%
%
\begin{claim}\label{cla:Jordan1}
Let $\epsilon,\delta_0,r_0$ and $\mu$ be as above.
For any $\epsilon'>0$ there exists a Legendrian embedding $F=(F_1,F_2,F_3)\colon\cd\hra \c^3$ of class $\Oscr(\cd)$ 
satisfying the following conditions.
\begin{enumerate}[\rm ({A}1)]
\item $\|F-f\|_{1,r_0\cd}<\epsilon'$.
\vspace{1mm}
\item $|F(u)-f(u')|<\epsilon$ and $|F(u)-\Ygot(u')|<\delta_0$ for all $u,u'\in\alpha_j$ and $j\in\z_m$.
\vspace{1mm}
\item Setting $G=(G_1,G_2,G_3):=F-\Ygot\colon\t\to\c^3$, the function $G_1G_3\colon\t\to\c$ has no zeros.
\vspace{1mm}
\item If $\gamma\subset\d$ is an arc with the initial point in 
$r_0\cd$ and the final point $u_j\in\T$ (see \eqref{eq:endpoints})
for some $j\in\z_m$, and if $\{J_a\}_{a\in\z_m}$ is any partition of $\gamma$ by Borel measurable subsets, then
\[
     \sum_{a\in\z_m}\length \big( \pi_{G(u_a)}(F(J_a)) \big)>2\mu.
\]
\end{enumerate}
\end{claim}
\begin{proof}
For every $j\in\z_m$, let $u_j\in\T$ be given by \eqref{eq:endpoints} and set 
\[
	\gamma_j:=\{t u_j : t\in[1,2]\}\subset \c.
\]
In $\c$ we consider the admissible subset 
\[
     S:=\cd\cup\big( \bigcup_{j\in\z_m} \gamma_j\big).
\]
Choose a number $c>0$ to be specified later. Let $\wt h=(\wt h_1,\wt h_2,\wt h_3) : S\to\c^3$ be a generalized Legendrian 
curve satisfying the following requirements:
\begin{enumerate}[\rm ({a}1)]
\item $\wt h=f$ on $\cd$;
\vspace{1mm}
\item $|\wt h(u)-f(u_j)|<c$ for all $u\in\gamma_j$, $j\in\z_m$;
\vspace{1mm}
\item $(\wt h_1(u)-\Ygot_1(u_j))(\wt h_3(u)-\Ygot_3(u_j))\neq 0$ for all $u\in\gamma_j$, $j\in\z_m$;
\vspace{1mm}
\item If $\{J_a\}_{a\in\z_m}$ is any partition of $\gamma_j$, $j\in \z_m$, 
by Borel measurable subsets, then
\[
     \sum_{a\in\z_m}\length \big( \pi_{g(u_a)}(\wt h(J_a)) \big) >2\mu.
\] 
\end{enumerate}
The existence of such $\wt h$ is ensured by \eqref{eq:f1-Y1}, \eqref{eq:f-Y}, and the fact that every compact path in $\c^3$ may be 
uniformly approximated by Legendrian paths (see Theorem \ref{th:path-approx}). Indeed, by \eqref{eq:f1-Y1} we have that $g(u_j)\neq 0$ 
for all $j\in\z_m$ and hence there is $w_0\in\c^3$ such that $\langle w_0,g(u_j)\rangle\neq 0$ for all $j\in\z_m$. (For instance, 
$w_0=(1,0,0)$ or $w_0=(0,0,1)$ meet this requirement since $g_1(u_j)g_3(u_j)\neq 0$ for all $j\in\z_m$.) 
Thus, it suffices to define $\wt h$ on each arc $\gamma_j$ as a Legendrian path which is highly oscillating in the direction of $w_0$ 
(which ensures condition {\rm (a4)}) but with very small diameter in $\c^3$ (which guarantees conditions {\rm (a2)} and {\rm (a3)}; 
take into account \eqref{eq:f1-Y1} for the latter). Choosing $c>0$ sufficiently small, \eqref{eq:f-Y} and condition {\rm (a2)} give that
\begin{enumerate}[\rm ({a}1)]
\item[\rm (a5)] $|\wt h(u)-f(u')|<\epsilon$ and $|\wt h(u)-\Ygot(u')|<\delta_0$ for all 
$(u,u')\in(\gamma_{j-1}\cup\alpha_j\cup\gamma_j)\times \alpha_j$, $j\in\z_m$.
\end{enumerate}

By Lemmas \ref{lem:fixedcomponents} and \ref{lem:position} we may approximate $\wt h$ in the $\Cscr^1(S)$-topology by Legendrian
embeddings $h=(h_1,h_2,h_3)\colon 3\cd\hra \c^{2n+1}$. If the approximation of $\wt h$ by $h$ is close enough, conditions {\rm (a4)} 
and {\rm (a5)} guarantee the existence of numbers 
\begin{equation}\label{eq:numbers}
       0<\tau<\lambda'<\lambda<\rho<1-r_0
\end{equation}
such that the following properties are satisfied (see Figure \ref{fig:Claim64}):
\begin{enumerate}[\rm ({b}1)]
\item $(\gamma_j+\tau\cd)\cap\cd\subset u_j+\lambda'\cd$, $j\in \z_m$;
\vspace{1mm}
\item $\big( (\gamma_j+\tau\cd)\cup (u_j+\lambda\cd) \big) \cap \big((\gamma_i+\tau\cd)\cup(u_i+\lambda\cd) \big)=\emptyset$ 
for all $i\neq j\in\z_m$;
\vspace{1mm}
\item $|h(u)-f(u')|<\epsilon$ and $|h(u)-\Ygot(u')|<\delta_0$ for all 
$(u,u')\in\big((\gamma_{j-1}+\tau\cd)\cup (u_{j-1}+\lambda\cd)\cup\alpha_j\cup(\gamma_j+\tau\cd)\cup (u_j+\lambda\cd)\big)\times \alpha_j$, 
$j\in\z_m$;
\vspace{1mm}
\item $(h_1(u)-\Ygot_1(u'))(h_3(u)-\Ygot_3(u'))\neq 0$ for all 
$(u,u')\in \big((u_j+\lambda\cd)\cup(\gamma_j+\tau\cd)\big)\times \big((u_j+\lambda\cd)\cap\t\big)$, $j\in\z_m$;
\vspace{1mm}
\item if $\gamma_j'\subset(u_j+\lambda\d)\cup(\gamma_j+\tau\cd)$ is an arc with the initial point in $u_j+\lambda\d$ and the final point 
$2u_j$, and if $\{J_a\}_{a\in\z_m}$ is any partition of $\gamma_j'$, $j\in \z_m$, by Borel measurable subsets, then
\[
     \sum_{a\in\z_m}\length \big(\pi_{g(u_a)}(h(J_a)) \big)>2\mu;
\]
\item $|h(u)-f(u')|<c$ for all $(u,u')\in \big((u_j+\lambda\cd)\cup(\gamma_j+\tau\cd)\big)\times \big((u_j+\lambda\cd)\cap\t\big)$, $j\in\z_m$.
\end{enumerate}
\begin{figure}[ht]
    \begin{center}
 \scalebox{0.16}{\includegraphics{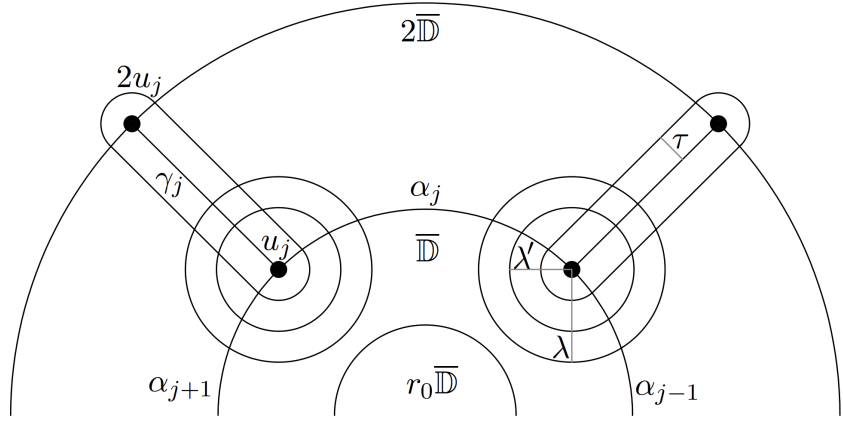}}
     \end{center}
\caption{Sets in the proof of Claim \ref{cla:Jordan1}.}
\label{fig:Claim64}
\end{figure}

Now, by Forstneri\v c and Wold \cite[Theorem 2.3]{ForstnericWold2009JMPA}, there is a smooth diffeomorphism 
$\phi\colon\cd\to\phi(\cd)$ satisfying the following conditions:
\begin{enumerate}[\rm ({c}1)]
\item $\phi(\cd)\subset S+\rho\d$.
\vspace{1mm}
\item $\phi|_\d\colon\d\to\phi(\d)$ is a biholomorphism.
\vspace{1mm}
\item $\phi$ is as close as desired to the identity map in the $\Cscr^1\big(\cd\setminus\bigcup_{j\in\z_m} (u_j+\lambda'\cd)\big)$-topology.
\item $\phi(u_j)=2u_j$ and $\phi\big(\cd\cap (u_j+\lambda'\cd)\big)\subset (u_j+\lambda\cd)\cup(\gamma_j+\tau\cd)$, $j\in\z_m$.
\end{enumerate}

If the approximation of $\wt h$ by $h$ and the one in condition {\rm (c3)} are close enough, then the Legendrian map $F:=h\circ \phi\colon\cd\to\c^3$,
which may be assumed to be an embedding of class $\Oscr(\cd)$ by Lemmas \ref{lem:fixedcomponents} and \ref{lem:position}, meets 
the conclusion of the claim. Indeed, condition {\rm (A1)} clearly follows if the mentioned approximations are sufficiently close. Moreover, 
condition {\rm (A2)} is ensured by conditions {\rm (b3)} and {\rm (c4)}. If $c>0$ is chosen sufficiently small, then, in view of 
conditions {\rm (b6)} and {\rm (c4)}, 
$G=F-\Ygot$ is so close to $g=f-\Ygot$ on $\t$ that \eqref{eq:f1-Y1} ensures condition {\rm (A3)}. Likewise, if $\gamma\subset\d$ is an arc 
with the initial point in $r_0\cd$ and the final point in $u_j$ for some $j\in\z_m$, and if $\{J_a\}_{a\in\z_m}$ is any partition of $\gamma$ 
by Borel measurable subsets, then conditions {\rm (b5)}, {\rm (c3)}, and {\rm (c4)} guarantee that
\[
     \sum_{a\in\z_m}\length \big( \pi_{g(u_a)}(F(J_a)) \big)>2\mu,
\]
and, if $c>0$ is sufficiently small, $G(u_a)$ is so close to $g(u_a)$ for all $a\in\z_m$ so that the above inequality gives condition {\rm (A4)}. 
This concludes the proof.
\end{proof}

Fix $\epsilon'>0$ to be specified later and let $F$ be given by Claim \ref{cla:Jordan1}. By condition {\rm (A4)}, there exists a number $r$
with $0<r<1-r_0$ enjoying the following property.
\begin{enumerate}[\rm ({A}1)]
\item[\rm ({A}5)] If $\gamma\subset\d$ is an arc with the initial point in $r_0\cd$ and the final point in $u_j+r\cd$ for some $j\in\z_m$, 
and if $\{J_a\}_{a\in\z_m}$ is any partition of $\gamma$ by Borel measurable subsets, then
\[
     \sum_{a\in\z_m}\length \big(\pi_{G(u_a)}(F(J_a)) \big)>2\mu.
\]
\end{enumerate}

Fix another constant $\epsilon''>0$ to be specified later. In view of {\rm (A2)}, the continuity of $F$ provides numbers 
$0<\varrho'<\varrho<r$  such that, setting for each $j\in\z_m$
\begin{eqnarray*}
    \alpha_j' & := &\alpha_j\setminus \big((u_{j-1}+\varrho\d)\cup (u_j+\varrho\d)\big) 
    \; \subset \; \alpha_j \; \subset \; \t,
    \\
    D_j & : = & (\alpha_j'+\varrho'\cd)\cap\cd \; \subset \; \cd\setminus r_0\cd
\end{eqnarray*}
\begin{figure}[ht]
    \begin{center}
 \scalebox{0.11}{\includegraphics{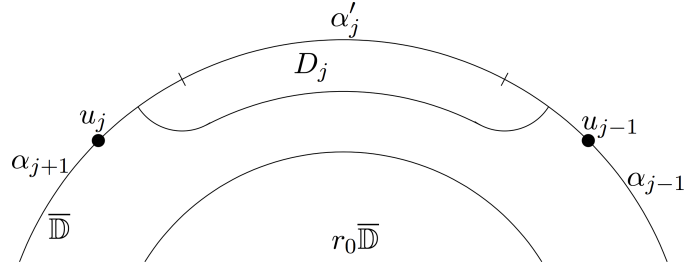}}
     \end{center}
\caption{The sets $\alpha_j'$ and $D_j$.}
\label{fig:Dj}
\end{figure}
(recall that $r<1-r_0$ and see Figure \ref{fig:Dj}), the following conditions are satisfied:
\begin{enumerate}[\rm ({B}1)]
\item $|F(u)-f(u')|<\epsilon$ and $|F(u)-\Ygot(u')|<\delta_0$ for all $(u,u')\in D_j\times\alpha_j$, $j\in\z_m$.
\vspace{1mm}
\item $|F(u/|u|)-F(u)|<\epsilon''$ for all $u\in D_j$, $j\in\z_m$.
\end{enumerate}

Now, for each $j\in\z_m$ let $\Pi_j\subset \c^3$ denote the vectorial complex $2$-plane Hermitian orthogonal to the vector 
\[
	G(u_j)=F(u_j)-\Ygot(u_j)\in\c^3\setminus\{(x,y,z)\in\c^3 : xz=0\}
\]
(see condition {\rm (A3)}):
\begin{eqnarray*}
     \Pi_j & = & \Big\{(V_1,V_2,V_3)\in\c^3 : 
     \overline{G_1(u_j)} V_1+\overline{G_2(u_j)} V_2+\overline{G_3(u_j)} V_3=0\Big\}
     \\
     & = & \Big\{(V_1,V_2,V_3)\in\c^3 : 
     \overline{G_1(u_j)/G_3(u_j)} V_1+\overline{G_2(u_j)/G_3(u_j)} V_2+ V_3=0\Big\}.
\end{eqnarray*}
Since $G_1(u_j)/G_3(u_j)\neq 0$ (see condition {\rm (A3)}), Proposition \ref{pro:flat} furnishes a continuous map 
$\Psi_j\colon (D_j\cap\t)\times\c\to\c^3$ satisfying the following conditions for all $u\in D_j\cap\t$:
\begin{itemize}
\item $\Psi_j(u,0)=F(u)$.
\vspace{1mm}
\item $\Psi_j(u,\cdot)\colon\c\to\c^3$ is a proper holomorphic Legendrian embedding.
\vspace{1mm}
\item $\Psi_j(u,\c)\subset F(u)+\Pi_j$.
\end{itemize}
Thus, by the Maximum Principle, suitably shrinking and reparametrizing the curves $\Psi_j(u,\cdot)$, $u\in D_j\cap\t$, 
we obtain a continuous map $H_j\colon (D_j\cap\t)\times\cd\to\c^3$ enjoying the following properties:
\begin{enumerate}[\rm ({C}1)]
\item $H_j(u,0)=F(u)$ for all $u\in D_j\cap\t$.
\vspace{1mm}
\item $H_j(u,\cdot)\colon\cd\to\c^3$ is a Legendrian disk of class $\Oscr(\cd)$ for all $u\in D_j\cap\t$.
\vspace{1mm}
\item $H_j(u,\cd)\subset F(u)+\Pi_j$ for all $u\in D_j\cap\t$.
\vspace{1mm}
\item $|H_j(u,\zeta)-F(u)|=\mu$ for all $u\in \alpha_j'$ and $\zeta\in\t$.
\vspace{1mm}
\item $|H_j(u,\zeta)-F(u)|\le\mu$ for all $u\in D_j\cap\t$ and $\zeta\in\cd$.
\vspace{1mm}
\item If $u\in \t$ is an endpoint of the arc $D_j\cap\t$ then $H_j(u,\zeta)=F(u)$ for all $\zeta\in\cd$.
\end{enumerate}
Consider the continuous map $H\colon\t\times\cd\to\c^3$ given by
\[
      H(u,\zeta)=\left\{
      \begin{array}{ll}
      H_j(u,\zeta) & \text{if $u\in D_j\cap\t$ for some $j\in\z_m$,}
      \\
      F(u) & \text{if $u\in \t\setminus\bigcup_{j\in\z_m} D_j$}.
      \end{array}
      \right.
\]

Given $\epsilon'''>0$ to be specified later, Lemma \ref{lem:RH} furnishes a Legendrian disk $\wt f\colon\cd\to\c^3$ of class 
$\Ascr^1(\cd)$ such that the following conditions hold:
\begin{enumerate}[\rm ({D}1)]
\item $\dist(\wt f(u),H(u,\t))<\epsilon'''$ for all $u\in\t$.
\vspace{1mm}
\item $\dist(\wt f(u),H(u/|u|,\cd))<\epsilon'''$ for all $u\in \bigcup_{j\in\z_m} D_j$.
\vspace{1mm}
\item $\|\wt f-F\|_{1,\cd\setminus \bigcup_{j\in\z_m} D_j}<\epsilon'''$.
\end{enumerate}
Moreover, by Lemmas \ref{lem:fixedcomponents} and \ref{lem:position}, we may assume that $\wt f$ is a Legendrian embedding 
of class $\Oscr(\cd)$. 

We claim that if the positive numbers $\epsilon$, $\epsilon'$, $\epsilon''$, and $\epsilon'''$ are chosen sufficiently small, the 
Legendrian embedding $\wt f$ satisfies the conclusion of the lemma.  Indeed, {\rm (A1)}, {\rm (D3)}, and the fact that 
$r_0\cd\subset \cd\setminus \bigcup_{j\in\z_m} D_j$ ensure that $\|\wt f-f\|_{1,r\cd}<\epsilon'+\epsilon'''<\epsilon_0$, provided 
that $\epsilon'>0$ and $\epsilon'''>0$ are so small to satisfy the latter inequality. To check condition {\rm (I)} pick a point $u\in \t$. 
Assume first that $u\in \t\setminus \bigcup_{j\in\z_m} D_j$. Then,
\begin{eqnarray*}
     |\wt f(u)-\Ygot(u)| & \stackrel{{\rm (D3)}}{<} & |F(u)-\Ygot(u)|+\epsilon'''
     \\
     & \stackrel{{\rm (A2)}}{<} & |f(u)-\Ygot(u)|+\epsilon'''+\epsilon'
     \\
     & \stackrel{\eqref{eq:Jordandelta0}}{<} & \delta_0+\epsilon'''+\epsilon' \; <\; \sqrt{\delta^2+\mu^2},
\end{eqnarray*}
where the last inequality holds provided that $\epsilon'>0$ and $\epsilon'''>0$ are small enough. Assume that, on the contrary, 
$u\in \t\cap D_j$ for some $j\in\z_m$. In this case, there exists $v_u\in\t$ and $w_u\in\Pi_j$ with $|w_u|\le\mu$ such that
\begin{eqnarray*}
     |\wt f(u)-\Ygot(u)| & \stackrel{{\rm (D1)},\eqref{eq:f-Y}}{<} & |H(u,v_u)-\Ygot(u_j)|+\epsilon'''+\epsilon
     \\
     & = & |H_j(u,v_u)-\Ygot(u_j)|+\epsilon'''+\epsilon
     		 \\
     & \stackrel{{\rm (C3)}, {\rm (C5)}}{=} & |F(u)+w_u-\Ygot(u_j)|+\epsilon'''+\epsilon
		 \\
     & \stackrel{{\rm (A2)},\eqref{eq:f-Y}}{\le} & |F(u_j)+w_u-\Ygot(u_j)|+\epsilon'''+4\epsilon
     \\
     & \stackrel{{\rm (C3)}}{=} & \sqrt{|F(u_j)-\Ygot(u_j)|^2+|w_u|^2}+\epsilon'''+4\epsilon
     \\
     & \stackrel{{\rm (A2)}}{<} & \sqrt{\delta_0^2 + \mu^2}+\epsilon'''+4\epsilon \; \stackrel{\eqref{eq:Jordandelta0}}{<} \; \sqrt{\delta^2+\mu^2}
\end{eqnarray*}
where the last inequality holds assuming that $\epsilon>0$ and $\epsilon'''>0$ are sufficiently small. This proves {\rm (I)}. 

Finally, let us check condition {\rm (II)}. If $\epsilon'>0$ and $\epsilon'''>0$ are chosen small enough, \eqref{eq:Jordanr}, condition {\rm (A1)}, 
and {\rm (D3)} guarantee that $\dist_{\wt f}(0,r_0\t)>d_0>d$, and hence it suffices to prove that $\dist_{\wt f}(r_0\t,\t)\ge \mu-(d_0-d)$;
equivalently, $\length(\wt f(\gamma))\ge \mu-(d_0-d)$ for all paths $\gamma\subset \cd\setminus r_0\d$ with the initial point in $r_0\t$ 
and the final point in $\t$. Let $\gamma$ be such a path. 

Assume first that $\gamma\cap (u_j+r\cd)=\emptyset$ for all $j\in\z_m$. In this case there is $j\in\z_m$ and a subarc $\wt\gamma$ of 
$\gamma$ such that $\wt\gamma\subset D_j\setminus\big((u_{j-1}+r\cd)\cup (u_j+r\cd) \big)$, the initial point $u$ of $\wt\gamma$ lies in 
$(bD_j)\cap\d$ and its final point $u'$ lies in $\alpha_j'\subset D_j\cap\t$. We have
\begin{eqnarray*}
     \length\bigl(\wt f(\gamma)\bigr) & > & \length\bigl(\wt f(\wt \gamma)\bigr)
     \\
     & \ge & |\wt f(u)-\wt f(u')|
     \\
     & \stackrel{{\rm (D3),\,(D1)}}{\ge} & |F(u)-H(u',v_u)|-2\epsilon'''\quad \text{for some $v_u\in\t$}
		\\
     & \stackrel{\eqref{eq:f-Y},{\rm (A2)}}{\ge} & |F(u')-H(u',v_u)|-2\epsilon'''-3\epsilon
		\\
     & \stackrel{{\rm (C4)}}{=} & \mu-2\epsilon'''-3\epsilon \; > \; \mu-(d_0-d),
\end{eqnarray*} 
where the last inequality holds provided that $\epsilon>0$ and $\epsilon'''>0$ are chosen small enough. 

Assume now that $\gamma\cap (u_j+r'\cd)\neq\emptyset$ for some $j\in\z_m$. Note that, since $\Pi_a$ is perpendicular to 
the vector $G(u_a)=F(u_a)-\Ygot(u_a)$ for all $a\in\z_m$, there exists for each $u\in D_a$ a point $v=v_u\in\cd$ such that
\begin{multline}\label{piG-piG}
       \big|\pi_{G(u_a)}(\wt f(u))-\pi_{G(u_a)}(F(u))\big|  = \big|\pi_{G(u_a)}(\wt f(u)-F(u)) \big|
	\\
	\stackrel{{\rm (D2)},\, {\rm (B2)}}{<} 
       \Big|\pi_{G(u_a)}\Big(H\big(\frac{u}{|u|},v\big)-F\big(\frac{u}{|u|}\big)\Big)\Big|+
       \epsilon''+\epsilon'''  \stackrel{{\rm (C3)}}{=}  \epsilon''+\epsilon'''.  
\end{multline}
Then, we have
\begin{eqnarray*}
       \length\bigl(\wt f(\gamma)\bigr) & = &   \length\bigl(\wt f(\gamma\setminus \bigcup_{a\in \z_m} D_a)\bigr) +
       \sum_{a\in\z_m}\length\bigl(\wt f(\gamma \cap D_a)\bigr)
       \\
       & \ge &  \length\bigl( \wt f(\gamma\setminus \bigcup_{a\in \z_m} D_a) \bigr)+
       \sum_{a\in\z_m}\length\bigl(\pi_{G(u_a)} (\wt f(\gamma \cap D_a))\bigr),
\end{eqnarray*}
and hence, in view of condition {\rm (D3)} and \eqref{piG-piG}, the length of $\wt f(\gamma)$ is greater or equal than
\[
	\length\bigl(F(\gamma\setminus \bigcup_{a\in \z_m} D_a)\bigr) +
       \sum_{a\in\z_m}\length\bigl(\pi_{G(u_a)} (F(\gamma \cap D_a))\bigr) - O(\epsilon''+\epsilon''').
\]
Thus, condition {\rm (A5)} ensures that $\length\bigl(\wt f(\gamma)\bigr)>\mu$ provided that $\epsilon''>0$ and $\epsilon'''$ are sufficiently small. 
This proves condition {\rm (II)} and concludes the proof of Lemma \ref{lem:Jordan0}.
\end{proof}

%
%

With Lemma \ref{lem:Jordan0} in hand, one can prove the following approximation result in the same way that 
\cite[Lemma 4.2]{AlarconDrinovecForstnericLopez2015PLMS} enables to prove 
\cite[Lemma 4.1]{AlarconDrinovecForstnericLopez2015PLMS}. As above, although we state it just for the disk, 
the next lemma also holds for any compact bordered Riemann surface.

%
%
\begin{lemma}\label{lem:Jordan}
Let $f\colon \cd\to\c^{2n+1}$ $(n\in\n)$ be a Legendrian curve of class $\Ascr^1(\cd)$. Given $\lambda>0$, $f$ may be approximated
uniformly on $\cd$ by Legendrian embeddings $\wt f\colon\cd\hra \c^{2n+1}$ of class $\Oscr(\cd)$ such that 
$\dist_{\wt f}(0,\t)>\lambda$.
\end{lemma}

The main point in the proof of Lemma \ref{lem:Jordan} is that, for any given constants $d_0>0$, $\delta_0>0$, and $c>0$, 
the sequence $\displaystyle d_j:=d_{j-1}+\frac{c}{j}$ diverges whereas the sequence 
$\displaystyle \delta_j:=\sqrt{\delta_{j-1}^2+\frac{c^2}{j^2}}$ converges as $j\in\n$ goes to infinity; this allows to approximate the 
initial curve $f$ uniformly on $\cd$ as close as desired by a Legendrian embedding $\wt f$ whose intrinsic boundary distance 
from $0\in\d$ is as large as desired.

Finally, Theorem \ref{th:Jordan} follows from Lemma \ref{lem:Jordan} by a standard recursive application. We refer to the proof of 
\cite[Theorem 1.1]{AlarconDrinovecForstnericLopez2015PLMS} via \cite[Lemma 4.1]{AlarconDrinovecForstnericLopez2015PLMS} 
and leave the details to the interested reader.


\appendix

\section{Holomorphic version of Darboux's theorems}

In this appendix we collect some results concerning holomorphic contact and symplectic forms and structures.
The corresponding results in the smooth case are well-known and can be found in numerous sources; 
however, their complex (holomorphic) versions do not seem easily available in the literature.
We do not claim any originality whatsoever since the proofs follow rather closely
those in the smooth case. In the present paper, we strongly use Theorem \ref{th:Darboux-contact}
(Darboux's theorem for holomorphic contact structures) and Theorem \ref{th:path-approx}
concerning the approximation by Legendrian paths.

%
%
\begin{theorem}\label{th:Darboux-symplectic}
Let $\omega$ be a closed nondegenerate holomorphic 2-form (i.e., a holomorphic symplectic form) on a 
complex manifold $M$ of even dimension $2n$. At every point $p\in M$ there exist local holomorphic 
coordinates $(x_1,y_1,\ldots,x_n,y_n)$ in which $\omega$ equals the standard holomorphic symplectic form 
$\omega_0=dx_1\wedge dy_1+\ldots + dx_n\wedge dy_n$.
\end{theorem}

\begin{proof}
We follow Moser's proof for the smooth case (see \cite{Moser1965TAMS}).
We may assume that $M=\C^{2n}$ with complex coordinates $z=(x_1,y_1,\ldots,x_n,y_n)$, $p=0$, and 
\[
	\omega|_{0}= \omega_0:=dx_1\wedge dy_1+\ldots + dx_n\wedge dy_n.
\]
Consider the family of holomorphic symplectic forms near $0\in\C^{2n}$ given by
\[
	\omega_t=(1-t)\omega_0 + t\omega,\quad t\in [0,1].
\] 
We wish to find a holomorphic vector field $V_t$ on a neighborhood of the origin in $\C^{2n}$ 
whose local holomorphic flow $\phi_t$ (the solution of $\dot \phi_t= V_t\circ \phi_t$, $\phi_0(z)=z$) satisfies
\begin{equation}\label{eq:flow}
	\phi_t^* \omega_t = \omega_0, \quad t\in [0,1]
\end{equation}
in some neighborhood of the origin. Since $\phi_0=\mathrm{Id}$, this holds at $t=0$. 
At time $t=1$ we shall then get  $\phi_1^*\omega = \omega_0$ which will prove the theorem. 

Let $\Lcal_V$ denote the Lie derivative of a vector field $V$. We differentiate  \eqref{eq:flow} on $t$:
\begin{equation}\label{eq:ddt0}
	0= \frac{d}{dt} (\phi_t^* \omega_t) = \phi_t^*\left(\dot{\omega_t} + \Lcal_{V_t}\omega_t  \right) 
	= \phi_t^*\left( \omega-\omega_0 + \Lcal_{V_t}\omega_t \right).
\end{equation}
Applying Cartan's formula for the Lie derivative and noting that $d\omega_t=0$ gives
\begin{equation}\label{eq:Cartan}
	\Lcal_{V_t}\omega_t = d\left(V_t\,\rfloor\, \omega_t\right) + V_t\,\rfloor\, d\omega_t = d\left(V_t\,\rfloor\, \omega_t\right). 
\end{equation}
The $2$-form  $\omega-\omega_0$ is closed and hence exact near the origin, $\omega-\omega_0=-d\lambda$
for some holomorphic $1$-form $\lambda$. Thus \eqref{eq:ddt0} is equivalent to
\[	
	0 = d\left( V_t \,\rfloor\, \omega_t - \lambda\right) \quad \text{for all $t\in [0,1]$}.
\]
This holds if the vector field $V_t$ is chosen such that
\[
	V_t \,\rfloor\, \omega_t = \lambda,\quad t\in [0,1].
\]
This algebraic equation for the coefficients of $V_t$ has a unique holomorphic solution.
\end{proof}

We have the following analogous result for holomorphic contact forms;
see e.g. \cite[Theorem 2.5.1, p.\ 67]{Geiges2008} for the smooth case. 

%
%
\begin{theorem}\label{th:Darboux-contact}
Let $\eta$ be a holomorphic contact form on a complex manifold $M^{2n+1}$. For every point $p\in M$
there exist local holomorphic coordinates $(x_1,y_1,\ldots,x_n,y_n,z)$ on a neighborhood of $p$ 
in which $\eta$ agrees with the standard contact form  
\begin{equation}\label{eq:eta0}
	\eta_0 = dz + \sum_{j=1}^n x_j \, dy_j.
\end{equation}
\end{theorem}

\begin{proof}
By the same argument as in the real case (see Geiges \cite{Geiges2008}), 
a holomorphic contact form $\eta$ uniquely determines a holomorphic 
vector field $R_\eta$ by the following two conditions:
\begin{equation}\label{eq:Reeb}
	R_\eta \,\rfloor\, \eta = \eta(R_\eta)=1,\quad R_\eta \,\rfloor\, d\eta=0.
\end{equation}
This $R_\eta$ is called the {\em Reeb vector field} of $\eta$. From the formula \eqref{eq:Cartan} we see that 
\[
	\Lcal_{R_\eta} \eta = d\left(R_\eta \rfloor\, \eta\right) + R_\eta\,\rfloor\, d\eta =0. 
\]

We may assume that $M=\C^{2n+1}$ and $p=0$ is the origin.
By a linear algebra argument, we can choose linear complex coordinates $(x_1,y_1,\ldots,x_n,y_n,z)$ 
on $\C^{2n+1}$ such that 
\[
	\text{$\eta=\eta_0$\ \ and\ \ $d\eta=d\eta_0$\ \ hold on $T_0\C^{2n+1}$.}
\]
It follows that 
\[
	\eta_t=\eta_0+t(\eta-\eta_0),\quad  t\in[0,1],
\]
is a smooth family of holomorphic contact forms on a neighborhood of $0\in\C^{2n+1}$
such that $\dot \eta_t=0$ holds at $0\in\C^{2n+1}$ for all $t\in[0,1]$. We shall find a time-dependent holomorphic 
vector field $V_t$ on a neighborhood of the origin on $\C^{2n+1}$ whose flow $\phi_t$ exists
on a smaller neighborhood of $0$ for all $t\in[0,1]$  and satisfies
\begin{equation}\label{eq:flow2}
	\phi_t^* \eta_t = \eta_0, \quad t\in [0,1].
\end{equation}	
At time $t=1$ we shall get  
\[
	\phi_1^*\eta = \eta_0
\]
which will prove the theorem.

Let $R_t$ denote the holomorphic Reeb vector field of $\eta_t$ (cf.\ \eqref{eq:Reeb}). We seek $V_t$ in the form 
\begin{equation}\label{eq:Vt}
	V_t=h_t R_t + Y_t,\quad Y_t\in \Lscr_t:= \ker(\eta_t),
\end{equation}
where $h_t$ is a smooth family of holomorphic functions and $Y_t\in \Lscr_t$ is a smooth family of 
holomorphic contact vector fields on a neighborhood of the origin. Then 
\[
	V_t \,\rfloor\, \eta_t = h_t,\quad  V_t\,\rfloor\, d\eta_t = Y_t \,\rfloor\, d\eta_t.
\]
Differentiating the equation \eqref{eq:flow2} on $t$ gives
\begin{equation}\label{eq:htYt}
	0= \dot{\eta_t} + d(V_t\rfloor \eta_t) + V_t\,\rfloor\, d\eta_t = (\eta-\eta_0 + dh_t) + Y_t \,\rfloor\, d\eta_t.
\end{equation}
Since $d\eta_t$ is nondegenerate on $\Lscr_t$, a suitable (unique!) choice of the vector field $Y_t$ tangent to 
$\Lscr_t$ ensures that $Y_t \,\rfloor\, d\eta_t$ equals any given holomorphic $1$-form
that is annihilated by $R_t$. Hence, it suffices to choose the function $h_t$ such 
that the component of the $1$-form $ \eta-\eta_0 + dh_t$ in the direction of $R_t$ vanishes.
This gives the following $1$-parameter family of quasilinear holomorphic partial differential equations for 
the functions $h_t$:
\[
	 R_t(h_t)  = R_t\rfloor dh_t  = R_t\rfloor (\eta_0-\eta),\quad t\in[0,1].
\]
Since $R_t$ is nonvanishing for all $t$ and the right hand side vanishes at $0\in\C^{2n+1}$, this equation
can be solved by the method of characteristics in a small neighborhood of $0$, and we can also
choose $h_t$ such that $h_t(0)=0$ for all $t\in[0,1]$.  Inserting the solution $h_t$ into \eqref{eq:htYt} 
we then obtain a unique holomorphic vector field $Y_t\in\Lscr_t$ such that the flow of the vector field 
$V_t$ given by \eqref{eq:Vt} satisfies condition \eqref{eq:flow2}. 
\end{proof}

A holomorphic vector field $V$ on a complex contact manifold $(M,\Lscr)$ is said to be 
an {\em infinitesimal automorphism} of the contact structure $\Lscr$, or a {\em contact holomorphic Hamiltonian}, 
if its local holomorphic flow $\phi_t$ preserves $\Lscr$, in the sense that for all $t\in\R$ we have $T(\phi_t)\Lscr=\Lscr$
on the maximal open subset $M_t\subset M$ on which the flow $\phi_\tau$ is defined for 
all $\tau\in [0,t]$. Assuming that $\Lscr=\ker(\eta)$, this is equivalent to 
\begin{equation}\label{eq:LVeta}
	\Lcal_V \eta= \lambda\eta\quad\text{for some}\ \ \lambda\in\Oscr(M).
\end{equation}
The following result describes infinitesimal automorphisms of a holomorphic contact structure;
see 
e.g.\ \cite[Theorem 2.3.1, p.\ 62]{Geiges2008} for the smooth case.

%
%
\begin{theorem}\label{th:infinitesimal}  
Let $(M,\eta)$ be a complex contact manifold and let $R$ be the associated Reeb vector field \eqref{eq:Reeb}. 
There is a bijective correspondence between holomorphic functions $h$ on $M$ and holomorphic 
vector fields $V$ on $M$ which are infinitesimal holomorphic automorphisms of the contact structure $\Lscr=\ker\eta$.
The correspondence is given by
\begin{itemize}
\item $V\mapsto h:=V\rfloor\, \eta \in \Oscr(M)$;
\vspace{1mm}
\item $\Oscr(M) \ni h \mapsto V$, where $V=V_h$ is uniquely determined by the conditions
\begin{equation}\label{eq:Vh}
	V\,\rfloor\, \eta = h, \quad V\,\rfloor\, d\eta = - dh + R(h)  \eta.
\end{equation}
\end{itemize}
In particular, if $M$ is compact then the only infinitesimal automorphisms of the contact structure
$\Lscr=\ker \eta$ are the constant multiples $cR$ $(c\in\C)$ of the Reeb vector field.
\end{theorem}

After the completion of the paper, we noticed that this result is already available in \cite[Proposition 2.1]{LeBrun1995IJM}.

\begin{proof}
Assume that $V$ is a contact Hamiltonian of $(M,\eta)$. 
Set $h=V\rfloor\eta\in \Oscr(M)$. By the Cartan formula for the Lie derivative, condition \eqref{eq:LVeta} is equivalent to
\[
	\lambda\eta = \Lcal_V \eta = d(V\rfloor\eta) + V\rfloor d\eta = dh + V\rfloor d\eta.
\]
Contracting this $1$-form by the Reeb vector field $R$ for $\eta$ gives 
\[
	\lambda = R \,\rfloor\, dh + R \,\rfloor\, (V\rfloor d\eta) = R(h) - V \,\rfloor\, (R\,\rfloor\, d\eta) = R(h).
\]
Inserting this into the previous formula shows that $V$ satisfies conditions \eqref{eq:Vh}.
Conversely, given a function $h\in \Oscr(M)$, the holomorphic $1$-form 
\[
	\alpha = -dh + R(h)\eta
\] 
clearly satisfies $R\,\rfloor\, \alpha=0$, so $\alpha$ has no component in the direction $R$. 
Since $\ker(d\eta)=\span(R)$, there exists a unique holomorphic vector field $Y$ on $M$ such that
\[	
	Y \,\rfloor\, \eta=0\quad \text{and}\quad Y \,\rfloor\, d\eta = \alpha. 
\]
Set $V=hR + Y$. Then $V\,\rfloor\, \eta = h$ and $V\,\rfloor\, d \eta = Y\,\rfloor\, d \eta =\alpha$, 
so $V$ satisfies condition \eqref{eq:Vh}. Cartan's formula shows that  
\begin{equation}\label{eq:LV}
	\Lcal_V \eta = d(V\rfloor \eta) + V\rfloor d\eta = dh + (-dh + R(h)\eta) = R(h) \eta. 
\end{equation}
Hence $V$ is an infinitesimal automorphism of the contact structure $\Lscr=\ker \eta$.
\end{proof}

The following corollary to Theorem \ref{th:infinitesimal}  is analogous to  \cite[Corollary 2.3.2, p.\ 63]{Geiges2008}.

\begin{corollary}
Let $(M,\eta)$ be a complex contact manifold. Given a smooth family of holomorphic functions 
$\{h_t\}_{t\in[0,1]}\subset \Oscr(M)$, let $\{V_t\}_{t\in[0,1]}$ be the corresponding family of contact Hamiltonians
defined by \eqref{eq:Vh}. Assume that the flow $\phi_t$ of the time-dependent vector field $V_t$ 
exists on a domain $M_0 \subset M$ for all $t\in[0,1]$. Then there is a smooth family of nonvanishing 
holomorphic function $\{\lambda_t\}_{t\in[0,1]} \subset \Oscr(M_0)$ such that
\[
	\phi_t^*\eta=\lambda_t\eta\quad\text{on}\ M_0,\ \  t\in[0,1].
\]
In particular, the biholomorphic map $\phi_t\colon M_0\to \phi_t(M_0)\subset M$ is a holomorphic contactomorphism
of the contact structure $\Lscr=\ker\eta$ for every $t\in[0,1]$.
\end{corollary}

\begin{proof}
By the assumption we have that $\phi_0=\Id_M$ and $\dot{\phi_t}=V_t\circ\phi_t$ on $M_0$ for all $t\in[0,1]$.
Let $R$ denote the Reeb vector field of $\eta$. By \eqref{eq:LV} we have 
$\Lcal_{V_t} \eta = R(h_t)\eta$ for all $t\in[0,1]$. Hence we get the following identity on $M_0$:
\[
	\frac{d}{dt} \phi_t^* \eta= \phi_t^* \left(\Lcal_{V_t} \eta\right) = \phi_t^* \bigl( R(h_t) \eta\bigr)
	= \mu_t \phi_t^*\eta,\quad t\in[0,1]
\]
where $\mu_t = R(h_t) \circ\phi_t\in \Oscr(M_0)$. Since $\phi_0^*\eta=\eta$,
it follows by integration that $\phi_t^* \eta=\lambda_t \eta$ where 
$\lambda_t=\exp\left(\int_0^t \mu_s ds\right) \in\Oscr(M_0)$ for $t\in[0,1]$.
\end{proof}

\begin{example}
Let $\eta=dz+\sum_{i=1}^n x_i dy_i$ be the standard contact form on $\C^{2n+1}$. 
Then the correspondence in Theorem \ref{th:infinitesimal}  is given by 
\[
	h\longmapsto V_h = \biggl( h - \sum_{j=1}^n x_j h_{x_j}\biggr) \di_z 
	+ \sum_{j=1}^n \bigl( (x_j h_z - h_{y_j}) \di_{x_j} + h_{x_j}\di_{y_j} \bigr).
\]
Note that for any linear function $h$ on $\C^{2n+1}$ the vector field $V_h$ 
is also linear and hence completely integrable. Its flow $\{\phi_t\}_{t\in\C}$ is a complex $1$-parameter family of 
$\C$-linear contactomorphisms of the standard contact structure $\Lscr= \ker\eta$ on $\C^{2n+1}$.
\end{example}

The last result that we mention concerns the possibility of approximating any smooth compact real curve in 
a complex contact manifold by Legendrian curves. 

%
%
\begin{theorem} \label{th:path-approx}
Let $(M,\Lscr)$ be a complex contact manifold.
Every path $\gamma\colon [0,1]\to M$ can be approximated in the $\Cscr^0$ topology by smooth 
embeddings $\lambda\colon [0,1]\to M$ tangential to $\Lscr$ (i.e., such that $\dot\lambda(t)\in \Lscr_{\lambda(t)}$ 
holds for all $t\in [0,1]$). Furthermore, assuming that the vectors $\dot \gamma(0)$  and $\dot \gamma(1)$ lie in 
$\Lscr$ at the respective points, $\lambda$ can be chosen such that $\dot\lambda(t)=\dot\gamma(t)$ for $t\in \{0,1\}$.
\end{theorem}

Since the contact distribution $\Lscr\subset TM$ is spanned by vector fields
which, together with their commutators, span the tangent bundle $TM$ at every point (see Subsec.\ \ref{ss:eta}),
Theorem \ref{th:path-approx} is essentially a corollary to the following theorem of Chow \cite{Chow1939MA} from 1939.

%
%
\begin{theorem}\label{th:Chow}
Let $V_1,\ldots,V_m$ be smooth vector fields on a connected
manifold $M$ such that their successive commutators span each tangent space $T_p M$, $p\in M$.
Then every two points in $M$ can be joined by a piecewise smooth path where each piece 
is a segment of an integral curve of one of these vector fields. Furthermore, every path $\gamma\colon [0,1] \to M$ 
can be $\Cscr^0$ approximated by piecewise smooth paths $\lambda\colon [0,1] \to M$ of the above type 
such that $\lambda(t)=\gamma(t)$ for $t\in\{0,1\}$. 
\end{theorem}

The approximation statement in Theorem \ref{th:Chow} is an immediate consequence of local connectivity 
by integral curves. Indeed, it suffices to subdivide the curve $\gamma$ into short arcs and connect the division points 
by integral curves lying in small connected open sets in $M$. 

Theorem \ref{th:Chow} has a complex origin. The basic case of vector fields tangent to the standard contact 
distribution on $\R^3$ was observed by Carath\'eodory (1909). The result was proved in essentially this form
by Chow \cite{Chow1939MA} in 1939; a similar result was obtained by Rashevski (1938). An informative historical
discussion can be found in Gromov's paper on the {\em Carnot-Carath{\'e}odory metrics}  
(see \cite[\textsection 0.2, p.\ 86]{Gromov1996PM});
these are metrics defined by curves tangent to a distribution spanned by a collection of vector fields. 
Gromov gave a proof of Chow's theorem in \cite[p.\ 113]{Gromov1996PM}, followed by a stronger quantitative version 
of it on p.\ 114. The proof can also be found in numerous other sources. Further, more precise results were obtained  
by Sussman \cite{Sussmann1973TAMS,Sussmann1973BAMS}. These notions also appear in optimal control theory and robotics
under the name of {\em controllability}; see the references in  \cite[\textsection0.2, p.\ 86]{Gromov1996PM}.

\begin{proof}[Proof of Theorem \ref{th:path-approx}]
By Theorem \ref{th:Chow}, $\gamma$ can be approximated in the $\Cscr^0$ topology
by piecewise smooth paths  $\wt\lambda\colon [0,1]\to M$ satisfying the conclusion of the theorem. 
We need to replace $\wt\lambda$ by a smooth embedding tangent to $\Lscr$
and matching $\gamma$ at the endpoints.

Let us first consider the case $M=\C^{2n+1}$ and $\Lscr =\ker\eta$, 
where $\eta$ is the standard holomorphic contact form given by \eqref{eq:eta0}.
Let $\wt\lambda=(\wt x_1,\wt y_1,\ldots,\wt x_n,\wt y_n,\wt z)\colon [0,1]\to \C^{2n+1}$.
By dimension reasons, a slight deformation of the map $(\wt x_1,\wt y_1,\ldots,\wt x_n,\wt y_n)\colon [0,1]\to\c^{2n}$ 
provides a smooth embedding $(x_1,y_1,\ldots,x_n,y_n)\colon [0,1]\to\c^{2n}$ such that, setting
\[
	z(t)= \wt z(0)-\int_0^t \sum_{j=1}^n x_j(s)\dot y_j(s)ds,\quad t\in [0,1],
\]
we have that 
\begin{itemize}
\item $z(t)\approx \wt z(t)$ for all $t\in[0,1]$,
\vspace{1mm}
\item $z(t)= \wt z(t)$ for all $t\in\{0,1\}$, and
\vspace{1mm}
\item if $\dot\gamma(0)$ and $\dot\gamma(1)$ lie in $\ker\eta$ at the respective points, then $\dot z(t)=\dot\gamma (t)$ for $t\in\{0,1\}$.
\end{itemize}
Thus, the smooth embedding $\lambda=(x_1,y_1,\ldots,x_n,y_n,z)\colon I\to\c^{2n+1}$ satisfies the conclusion of the theorem.

In the general case, we choose a division $0=t_0<t_1<\cdots t_k=1$ of $[0,1]$ such that for every $i=1,\ldots,k$
we have $\wt\lambda([t_{i-1},t_i])\subset U_i$, where $U_i\subset M$ is a connected coordinate neighborhood 
such that the restriction $\Lscr|_{U_i}$ is given by the contact form \eqref{eq:eta0} (cf.\ Theorem \ref{th:Darboux-contact}).
The above argument can then be applied within each $U_i$, making sure that the  
embedded Legendrian curves $\lambda_i\colon [t_{i-1},t_i]\to U_i$ $(i=1,\ldots,k)$ smoothly match at the 
respective endpoints and do not intersect elsewhere. 
\end{proof}


\subsection*{Acknowledgements}
A.\ Alarc\'on is supported by the Ram\'on y Cajal program of the Spanish Ministry of Economy and Competitiveness.
A.\ Alarc\'on and F.\ J.\ L\'opez are partially supported by the MINECO/FEDER grant no. MTM2014-52368-P, Spain. 
F.\ Forstneri\v c is partially  supported  by the research program P1-0291 and the grant J1-7256 from 
ARRS, Republic of Slovenia. 

The authors would like to thank Yakov Eliashberg for 
having provided some of the references related to Theorem \ref{th:Chow}, and Jun-Muk Hwang for 
useful information regarding the results on compact Legendrian submanifolds of compact complex contact manifolds.





\noindent Antonio Alarc\'{o}n

\noindent Departamento de Geometr\'{\i}a y Topolog\'{\i}a e Instituto de Matem\'aticas (IEMath-GR), Universidad de Granada, Campus de Fuentenueva s/n, E--18071 Granada, Spain.

\noindent  e-mail: {\tt alarcon@ugr.es}

\vspace*{0.3cm}
\noindent Franc Forstneri\v c

\noindent Faculty of Mathematics and Physics, University of Ljubljana, and Institute
of Mathematics, Physics and Mechanics, Jadranska 19, SI--1000 Ljubljana, Slovenia.

\noindent e-mail: {\tt franc.forstneric@fmf.uni-lj.si}

\vspace*{0.3cm}
\noindent Francisco J.\ L\'opez

\noindent Departamento de Geometr\'{\i}a y Topolog\'{\i}a e Instituto de Matem\'aticas (IEMath-GR), Universidad de Granada, Campus de Fuentenueva s/n, E--18071 Granada, Spain

\noindent  e-mail: {\tt fjlopez@ugr.es}

\end{document}